\documentclass[11pt,reqno]{amsart}%%%%%%%%%%
\usepackage {amssymb}
\usepackage {amsmath}
\usepackage {bbm}
\usepackage{amsthm}
\usepackage{mathtools}
\usepackage{graphicx}
\usepackage {amscd}
\usepackage {epic}
\usepackage[mathscr]{euscript}

\usepackage[dvipsnames]{xcolor}
\usepackage[colorlinks,citecolor=OliveGreen,linkcolor=Mahogany,urlcolor=Plum,pagebackref]{hyperref}
\usepackage[alphabetic]{amsrefs}

\usepackage{enumerate}
\usepackage{tikz}
\usepackage{tikz-cd}
\usepackage{verbatim,color,geometry}
\usepackage[all]{xy}
\usepackage{enumitem}
\usepackage{bm}
\newcommand{\A}{\mathbb{A}}
\newcommand{\G}{\mathbb{G}}

\newcommand{\N}{\mathbb{N}}

\newcommand{\Q}{\mathbb{Q}}
\newcommand{\R}{\mathbb{R}}
\newcommand{\T}{\mathbb{T}}
\newcommand{\Z}{\mathbb{Z}}

\newcommand{\cL}{{\mathcal{L}}}

\newcommand{\cF}{{\mathcal{F}}}
\newcommand{\cI}{{\mathcal{I}}}

\newcommand{\cX}{{\mathcal{X}}}

\newcommand{\Aut}[0]{\operatorname{Aut}}

\newcommand{\DMR}{{\mathcal{DMR}}}

\newcommand{\Supp}{{\rm Supp}}
\newcommand{\Hom}{{\rm Hom}}

\newcommand{\ord}{{\rm ord}}
\newcommand{\gr}{{\rm gr}}

\numberwithin{equation}{section}
\newtheorem{thm}{Theorem}[section]

\newtheorem{lem}[thm]{Lemma}
\newtheorem{cor}[thm]{Corollary}

\newtheorem{prop}[thm]{Proposition}

\theoremstyle{definition}
\newtheorem{defn}[thm]{Definition}

\newtheorem{example}[thm]{Example}

\newtheorem{que}[thm]{Question}    %!!!!!!!!!!!!!!!!!!!!

\newtheorem{rem}[thm]{Remark}

\newtheorem{defn-thm}[thm]{Definition--Theorem}  %!!!!!!!!!!!!!!!!!!!!!!!!
\newtheorem{defn-prop}[thm]{Definition--Proposition}  %!!!!!!!!!!!!!!!!!!!!!!!!
\newtheorem{defn-lem}[thm]{Definition--Lemma}  %!!!!!!!!!!!!!!!!!!!!!!!!
\theoremstyle{remark}

\newcommand{\Fut}{{\rm Fut}}

\newcommand{\Val}{{\rm Val}}

\newcommand{\vol}{{\rm vol}}

\newcommand{\cG}{{\mathcal{G}}}
\newcommand{\cO}{{\mathcal{O}}}

\newcommand{\fa}{{\mathfrak{a}}}
\newcommand{\lct}{{\rm lct}}
\newcommand{\im}{{\rm im}}
\newcommand{\bQ}{{\mathbb{Q}}}
\newcommand{\fb}{{\mathfrak{b}}}
\newcommand{\Spec}{{\rm Spec}}
\newcommand{\Dh}{{\rm DH}}
\newcommand{\RLM}{{\rm RLM}}

\newcommand{\la}{\lambda}

\newcommand{\bnu}{\bm{\nu}}

\newcommand{\Int}{{\rm Int}}

\newcommand{\HB}[1]{{\textcolor{red}{[Harold: #1]}}}
\newcommand{\CX}[1]{{\textcolor{red}{[Chenyang: #1]}}}
\newcommand{\YL}[1]{{\textcolor{blue}{[Yuchen: #1]}}}

\newcommand{\bZ}{\mathbb{Z}}
\newcommand{\bR}{\mathbb{R}}
\newcommand{\bC}{\mathbb{C}}
\newcommand{\bT}{\mathbb{T}}
\newcommand{\QM}{\mathrm{QM}}

\newcommand{\bA}{\mathbb{A}}
\newcommand{\bN}{\mathbb{N}}

\newcommand{\bG}{\mathbb{G}}
\newcommand{\bk}{\mathbbm{k}}

\newcommand{\tX}{\widetilde{X}}

\newcommand{\tS}{\widetilde{S}}
\newcommand{\tD}{\widetilde{D}}

\newcommand{\rd}{\mathrm{d}}

\newcommand{\wt}{\mathrm{wt}}

\newcommand{\Ex}{\mathrm{Ex}}
\newcommand{\triv}{\mathrm{triv}}

\newcommand{\tbeta}{\tilde{\beta}}

\newcommand{\bfV}{\mathbf{V}}

\newcommand{\Valc}{{\rm Val}_X^{\circ}}
\newcommand{\ValT}{{\rm Val}_X^{\bT,\circ}}

\newcommand{\HNA}{\mathbf{H}^{\mathrm{NA}}}
\newcommand{\LNA}{\mathbf{L}^{\mathrm{NA}}}
\newcommand{\ENA}{\mathbf{E}^{\mathrm{NA}}}
\newcommand{\DNA}{\mathbf{D}^{\mathrm{NA}}}
\newcommand{\JNA}{\mathbf{J}^{\mathrm{NA}}}

\title{The existence of the K\"ahler-Ricci soliton degeneration}
\date{\today}

\author{Harold Blum}
\address{Department of Mathematics, University of Utah, Salt Lake City, UT 84112, USA}
\email{blum@math.utah.edu}

\author{Yuchen Liu}
\address{Department of Mathematics, Northwestern University, Evanston, IL 60208, USA}
\email{yuchenl@northwestern.edu}

\author{Chenyang Xu}
\address{Department of Mathematics, Princeton University, Princeton, NJ 08544, USA}
\email     {chenyang@princeton.edu}
\address   {Beijing International Center for Mathematical Research,       Beijing 100871, China}
\email     {cyxu@math.pku.edu.cn}

\author{Ziquan Zhuang}
\address{Department of Mathematics, Princeton University, Princeton, NJ 08544, USA}
\email{zzhuang@princeton.edu}
\address{Department of Mathematics, Johns Hopkins University, Baltimore, MD 21218, USA}
\email{zzhuang@jhu.edu}

\begin{document}

\pagenumbering{arabic}

\maketitle

\begin{abstract}
    We prove an algebraic version of the Hamilton-Tian Conjecture for all log Fano pairs. More precisely, we show that  any log Fano pair  admits a canonical  two-step degeneration to a reduced uniformly Ding stable triple, which admits a K\"ahler-Ricci soliton when the ground field $\bk=\bC$.
\end{abstract}

\setcounter{tocdepth}{1}
\tableofcontents

Throughout this paper, we work over an algebraically closed field $\bk$ of characteristic $0$.

\section{Introduction}

A K\"ahler-Einstein metric is arguably `the most canonical' metric that one can find on a Fano variety. However,  not every Fano variety admits a K\"ahler-Einstein metric. So it is natural to ask what kind of structure one should look for on a general Fano variety. In fact, there are several candidates. In this note, we will study one structure, namely the {\it K\"ahler-Ricci soliton}. This kind of metric has been investigated in many previous works. %, see e.g. \cites{TZZZ13, He16, TZ16, Bamler, CW20} for an incomplete list.
While not every Fano variety itself admits a K\"ahler-Ricci soliton, it is expected that any Fano variety has a unique degeneration to one with a K\"ahler-Ricci soliton (see e.g. \cite{Tia97}*{Section 9}). 

For a smooth Fano manifold $X$, the approach of using K\"ahler-Ricci flow to study K\"ahler-Ricci solitons has been intensively studied in complex geometry literature
and lead to the solution of the Hamilton-Tian Conjecture (see \cites{TZ16, Bamler, CW20}), which says that the Gromov-Hausdorff limit $X_{\infty}$ of $X$ under the K\"ahler-Ricci flow admits a K\"ahler-Ricci soliton. What is more relevant to us is that in \cite{CSW18}  it is shown that $X_{\infty}$ can be realized as a (two-step) degeneration of $X$,  and in \cite{DS20} that the first degeneration minimizes the $H$-functional among all possible $\R$-degenerations (note that our sign on $H$ is opposite to the one in \cite{DS20}). 

In this paper, we will pursue a purely algebraic study of the above degeneration process, by studying the geometry of the minimizer of the non-Archimedean type functional $\HNA$, which in particular can be applied to a general (possibly singular) log Fano pair $(X,\Delta)$.
Such  algebraic study, including developing the non-Archimedean theory of the $\HNA$-functional, was initiated in \cite{HL20-uniqueness}. There it is shown the {\it uniqueness} of the above degeneration process. Our first main theorem of this paper is the {\it existence} of such a degeneration for general log Fano pairs. It can be considered to be an algebraic version of the Hamilton-Tian Conjecture, though there is no metric involved.  More precisely, we have the following result. 
\begin{thm}\label{t-optimaldeg}
Let $(X,\Delta)$ be a log Fano pair. Then it admits a two-step degeneration to a K-polystable triple $(Y,\Delta_Y,\xi_Y)$, which is indeed reduced uniformly Ding stable. In particular, it admits a K\"ahler-Ricci soliton if $\bk=\bC$.
\end{thm}
In the above theorem, $\xi_Y=0$ if and only if $(X,\Delta)$ is K-semistable. The proof of Theorem \ref{t-optimaldeg} will be separated into two parts, contained respectively in Theorems \ref{t-finitegen} and \ref{t-YTD}.

By \cite{HL20-uniqueness}*{Theorem 1.3}, it is already known that any K-semistable triple $(X_0,\Delta_0,\xi)$ admits a unique K-polystable degeneration $(Y,\Delta_Y,\xi_Y)$ (whose proof is based on \cite{LWX-tangentcone}).
Therefore it suffices to establish the first-step degeneration which degenerates $(X,\Delta)$ to a K-semistable triple $(X_0,\Delta_0,\xi)$.
To construct such a degeneration, we follow \cite{HL20-uniqueness} and study the valuation which computes $h(X,\Delta)$ (i.e. the minimizer of the $\tbeta$). As a result we prove the following statement, which establishes the first half of Theorem \ref{t-optimaldeg} and gives an affirmative answer to \cite{HL20-uniqueness}*{Conjecture 4.10}.

\begin{thm}\label{t-finitegen}
Let $(X,\Delta)$ be a log Fano pair. Let $r$ be a positive integer such that $r(K_X+\Delta)$ is Cartier. 
Then $(X,\Delta)$ has a unique valuation $v$ computing $h(X,\Delta)$. 
%\footnote{\CX{I change all places to computing $h(X,\Delta)$ or minimizing $\tbeta_{X,\Delta}$. $\HNA(v)$ will be different with $\HNA(\cF_v)$ which I find a little confusing. But if more people prefer, we can also use $\HNA(v)$.}}

Moreover, the associated graded ring ${\rm gr}_v R$ for $R=\bigoplus_{m\in \bN} H^0(-mr(K_X+\Delta))$ is finitely generated, whose ${\rm Proj}$ together with the degeneration of $\Delta$ and the induced vector $\xi_v$ yields the first step degeneration to a K-semistable triple $(X_0,\Delta_0,\xi_v)$.
\end{thm}

It was shown in \cite{HL20-uniqueness} that the uniqueness statement in the above theorem follows from the finite generation of the $\gr_v R$. In this note, we first establish stronger convexity results for various non-archimedean functionals (see Theorem \ref{t:DHconvex}). Then we will obtain uniqueness as a consequence, without using  finite generation.

\medskip

The second step to proving Theorem \ref{t-optimaldeg} is to establish the following Yau-Tian-Donaldson (YTD) Conjecture for K\"ahler-Ricci solitons.
\begin{thm}[{YTD Conjecture for K\"ahler-Ricci Solitons}]\label{t-YTD}
A  triple $(X,\Delta,\xi)$ is K-polystable if and only if it is reduced uniformly Ding stable. In particular, when $\bk=\bC$, $(X,\Delta,\xi)$ admits a K\"ahler-Ricci soliton if and only if it is K-polystable.
\end{thm}

In fact, in \cite{HL20-YTD}, it is proven that the reduced uniform Ding stability of $(X,\Delta,\xi)$ is equivalent to the existence of a  K\"ahler-Ricci soliton, by using variational methods. Here we verify that reduced uniform Ding stability is equivalent to K-polystability. When $X$ is smooth and $\Delta=0$, the second part of Theorem \ref{t-YTD} is proved in \cites{DS16, CSW18}.

\begin{rem}
As we already mentioned, when $X$ is smooth, the Cheeger-Colding-Tian theory can be used to establish Theorem \ref{t-optimaldeg}, Theorem \ref{t-finitegen} and Theorem \ref{t-YTD}. In fact one can obtain the optimal degeneration from the study of the Hamilton-Tian conjecture on the long time behavior of K\"ahler-Ricci flows on $X$. See \cites{DS16, TZ16, Bamler, CW20, CSW18}. However, it seems to us it is hard to extend these types of arguments to the more general (possibly singular) case.

Recent work of Han and Li builds an algebraic framework for studying the two-step degeneration process. Specifically, in \cite{HL20-uniqueness}, they developed the non-Archimedean theory for the $\HNA$-functional (based on the $H$-function(al) defined in \cite{TZZZ13, He16}), and interpret the existence of the optimal degeneration, which is essentially equivalent to Theorem \ref{t-finitegen}, in terms of geometric properties of the minimizer of the $\tbeta_{X,\Delta}$-function, which is a variant of the $\HNA$-functional but defined on $\Valc\cup\{v_{\triv}\}$.

From the algebro-geometric viewpoint, the study of $\tbeta_{X,\Delta}$ in \cite{HL20-uniqueness} is entirely parallel to the study of the stability threshold of a log Fano pair or in the local setting the normalized volume function of a klt singularity. Therefore, it is natural to apply the arguments in the former problems to the current case. 
Indeed, \cite{HL20-uniqueness} has made significant progress in carrying out this study, and the main remaining step is the finite generation of the graded ring induced by minimizers of various functions, e.g. $\tbeta_{X,\Delta}$ and $\delta_g(X,\Delta,\xi)$.

In \cite{LXZ-HRFG}, the finite generation of the associated graded ring for the valuation computing the stability threshold is solved. In this note, we  solve the finite generation in Theorem \ref{t-finitegen} and Theorem \ref{t-YTD} by a similar method.  We also give a more straightforward argument of the uniqueness without using the finite generation (which is needed in \cite{HL20-uniqueness}), but by establishing convexity of various functionals based on the arguments in \cite{XZ20nvol}.

\end{rem}

We will also investigate a moduli approach to study general log Fano pairs with fixed $h$-invariant.
%Corresponding to the two step degeneration, we also have two morphisms.
\begin{thm}[{=Theorem \ref{t-moduli}}]\label{t-stack}
For a fixed dimension $n$, volume $V$, a positive integer $N$ and a constant $h_0$, families of $n$-dimensional log Fano pairs $(X,\Delta)$ with $(-K_X-\Delta)^n=V$, $N\Delta$ integral and $h(X,\Delta)\ge h_0$ are parametrized by an Artin stack 
$\mathcal{M}^{\rm Fano}_{n,V,N,h^+_0}$ of finite type.  
\end{thm}

In the upcoming work, we aim to show that there is a finite type Artin stack $\mathcal{M}^{\rm Kss}_{n,V,N,h_0}$ which parametrizes families of $n$-dimensional K-semistable triples $(X,\Delta,\xi)$ with $(-K_X-\Delta)^n=V$, $N\Delta$ integral and $h(\xi)=h_0$. Moreover, $\mathcal{M}^{\rm Kss}_{n,V,N,h_0}$ admits a proper good moduli space ${M}^{\rm Kps}_{n,V,N,h_0}$. Then we will study the two-step degeneration from a  moduli theoretic viewpoint. 
%$$\mathcal{M}^{\rm Fano}_{n,V,N,h_0}\to \mathcal{M}^{\rm Kss}_{n,V,N,h_0}\to {M}^{\rm Kps}_{n,V,N,h_0},$$
%which captures the two-step degeneration process. 

\subsection{Outline of the proof}
In recent years, there have been two functions on the space of valuations which have been intensively studied in algebraic geometry. The first one is the function $\frac{A_{X,\Delta}(\cdot)}{S(\cdot)}$ of a log Fano pair $(X,\Delta)$; and the second one is the normalized volume function on a klt singularity $x\in (X,\Delta)$. Many of their fundamental properties were proved in a sequence of works. 
The general framework for the proofs of the theorems in this paper is largely parallel to the previous works on the study of these two functions, especially the first one. 

Step 1: The first step is to show the strict convexity of the $\HNA$-functional. For any pair of filtrations $\cF_0$ and $\cF_1$, there is a natural family $(\cF_t)_{t\in [0,1]}$  of filtrations connecting them called the {\it geodesic} (see Section \ref{ss:geodesic}). To study it, we define a measure ${\rm DH}_{\cF_0,\cF_1}$ over $\mathbb{R}^2$ (called {\it the compatible Duistermaat-Heckman measure}) that encodes $\Dh_{\cF_t}$ for $t\in [0,1]$ (see Section \ref{s:measureR2}). Then convexity results for various functionals, e.g $\ENA$ and $\tS$ along geodesics can be proved by doing integration over this measure. And for $\LNA$, the convexity is proved by interpreting it as the log canonical slope $\mu$ and then applying results from \cite{XZ20nvol} to compare log canonical thresholds. 
As a result, this yields the convexity of $\DNA$ and the strict convexity of $\HNA$ along geodesics (analogous results in the archimedean setting were proved in \cite{Ber15}). The latter will imply the uniqueness of the minimizer. 

Step 2: To prove Theorem \ref{t-YTD}, we will take a similar strategy to the solution of the usual YTD Conjecture for log Fano pairs.
First, we will extend the usual definition of the $S$-invariant function to the weighted setting with respect to a quasi-monomial valuation $v_0$ (see Section \ref{ss-weighteddelta}) and then we can define the corresponding $\delta(X,\Delta,v_0)$. 

We are first interested in the special case when $v_0$ is a valuation coming from a vector field $\xi$ induced by a torus action. In this case, as seen in \cite{HL20-uniqueness}, many criteria for testing the the K-semistability or (reduced) uniformly K-stability of a pair $(X,\Delta)$ can be extended in to the setting of triples $(X,\Delta,\xi)$. In particular, we extend results from \cite{Li19} and \cite{XZ20cm}*{Appendix} to this setting in Sections \ref{ss-reduceduniformK} and \ref{ss-reducedminimizer}.

 In a similar but slightly different setting, we consider the minimizer $v_0$ of the non-homogeneous function $\tbeta_{X,\Delta}$, and we show it computes $\delta(X,\Delta,v_0)$ for the weight function $g=e^{-x}$. See Section \ref{ss-HNAval}.

Step 3: In the last step, we will show in the above two cases $v_0$ is a monomial lc place of a special $\mathbb{Q}$-complement (in the sense of \cite[Definition 3.3]{LXZ-HRFG}) constructed from a weighted basis type divisor. Then we can apply the finite generation result in \cite{LXZ-HRFG} to show the associated graded ring of $v_0$ is finitely generated (see Corollary \ref{c-hfg} and \ref{c-deltafg}). This completes the proof of Theorem \ref{t-finitegen} and \ref{t-YTD}.
%and this completes the proof of Theorem \ref{t-finitegen}. Then we can apply the work in \cite{LXZ-HRFG} again to establish a finite generation theorem (see Corollary ) again to conclude that $(X,\Delta;\xi)$ is not K-polystable.

\medskip 

 Finally, to prove Theorem \ref{t-stack}, i.e. to verify $\mathcal{M}^{\rm Fano}_{n,V,N,h^{+}_0}$ is an Artin stack of finite type, we first need to show that the set of all $n$-dimensional log Fano pairs $(X,\Delta)$ with $(-K_X-\Delta)^n=V$, $N\Delta$ integral and $h(X,\Delta)\ge h_0$ is bounded; and then we conclude by showing for any $\mathbb{Q}$-Gorenstein family of $(X,\Delta)\to B$ over a finite type base $B$, the function $B\ni t \mapsto h(X_t, \Delta_t)$ is constructible and lower semicontinuous.

\bigskip

\noindent {\bf Acknowledgement}: We would like to thank Mattias Jonsson, Valentino Tosatti, and Xiaowei Wang for helpful conversations. We are also grateful to the anonymous referee for the valuable comments. 
HB is partially supported by NSF DMS-1803102 and DMS-2200690.
YL is partially supported by NSF DMS-2148266 (formerly DMS-2001317) and an Alfred P. Sloan
research fellowship. 
CX is partially supported by NSF DMS-2153115 (formerly DMS-1901849), DMS-2139613 (formerly DMS-1952531), DMS-2201349, and a Simons Investigator award.
ZZ is partially supported by NSF DMS-2240926 (formerly DMS-2055531) and a Clay research fellowship.

\section{Preliminaries}

\noindent {\bf Notation and Conventions:}  We follow the standard notation as in \cite{KM98, Kol13, Laz04}. 

A variety is a separated integral scheme of finite type over $\bk$. A pair $(X,\Delta)$ consists of a normal variety $X$ and an effective $\bQ$-divisor $\Delta$ on $X$ such that $K_X+\Delta$ is $\bQ$-Cartier. 
A pair $(X,\Delta)$ is called {\it log Fano} if $X$ is projective, $(X,\Delta)$ is klt, and $-K_X-\Delta$ is ample.
A log smooth model $(Y,E)$ of a pair $(X,\Delta)$ consists of a projective birational morphism $\pi\colon Y\to X$ and a reduced divisor $E$ on $Y$ such that $(Y,\Supp(E+\Ex(\pi)+\pi_*^{-1}\Delta))$ has simple normal crossing support.

Let $X$ be a normal variety.  We denote by $\Val_X$ the space of real valuations $K(X)^\times \to \R $ centered on $X$ whose restriction over the ground field $\bk$ is trivial. We endow $\Val_X$ with the weak topology.  We denote the trivial valuation on $X$ by $v_{\triv}$.

For the definitions of divisorial valuations, quasi-monomial valuations, and log discrepancy of valuations, see e.g. \cites{JM12, LLX20, Xu20}.

\begin{defn}
Let $(X,\Delta)$ be a pair. We denote by 
\[
\Val_X^\circ:=\{v\in \Val_X\mid A_{X,\Delta}(v) <+\infty\textrm{ and }v\neq v_{\triv}\}.
\] If $(X, \Delta)$ admits a torus $\bT$-action, then we denote by $\Val_X^{\bT}$  the subset of $\Val_X$ consisting of all $\bT$-invariant valuations, and let $\ValT:=\Val_X^\circ\cap \Val_X^{\bT}$. %We note that $v_{\triv}$ is not contained in $\Val_X$ by our convention. 
\end{defn}

\subsection{Filtrations}
Let $(X,\Delta)$ be an $n$-dimensional log Fano pair.
 Fix $r>0$ so that  $L:=-r(K_X+\Delta)$ is an ample Cartier divisor.
We write 
\[
R(X,L):=R
: = 
\bigoplus_{m \in \N} R_m
: =
\bigoplus_{m \in \N} H^0(X,\cO_X(mL))
\]
for the section ring of $L$ and $N_m := \dim R_m$.

\begin{defn}
A \emph{filtration} $\cF$ of $R$ is a collection of vector subspaces $\cF^\la R_m \subset R_m$ for each $\la \in \R$ and $m \in \N$ satisfying 
\begin{itemize}
\item[(F1)] $\cF^\la R_m \subset \cF^{\la'} R_m$ for $\la\geq \la'$;
\item[(F2)] $\cF^\la R_m = \bigcap_{\la' <\la} \cF^{\la'} R_m$;
\item[(F3)] $\cF^\la R_m = R_m$ for $\la \ll 0$ and $\cF^\la R_m=0$ for $\la\gg0$;
\item[(F4)] $\cF^\la R_m \cdot \cF^{\la'} R_{m'} \subset \cF^{\la+\la'} R_{m+m'}$.
\end{itemize}
\end{defn}
A filtration $\cF$ is a \emph{$\Z$-filtration} if $\cF^{\lceil \la \rceil} R_m= \cF^\la R_m$ for all $\la \in \R$ and $m \in \N$. A filtration is \emph{finitely generated} if the associated graded $\bk$-algebra
$\gr_\cF R:=\bigoplus_{(m,\la) \in \N\times \R} \gr_{\cF}^\la R_m $, where $\gr_{\cF}^\la R_m = \cF^\la R_m / (\cup_{\mu > \la} \cF^\mu R_m)$, is finitely generated.

The {\it translation} of $\cF$ by $c\in \mathbb{R}$ is the filtration defined by $\cG^\la R_m : = \cF^{\la-mrc} R_m$. 
The {\it scaling} by $a\in \bR_{>0}$ is the filtration defined by $\mathcal{H}^{\la} R_m := \cF^{\la/a}R_m$.

A filtration $\cF$ is \emph{linearly bounded} if there exists $C>0$ so that $\cF^{mC} R_m =0$ for all $m>0$. Note that there always exists $C>0$ so that $\cF^{-mC}R_m=R_m$ for all $m>0$ by  the finite generation of $R$ combined with (F3) and (F4). 

For an element $s\in R_m\setminus\{0\}$, we set
 $\ord_{\cF}(s): = \max \{ \la \in \R \, \vert \, s\in  \cF^\la R_m \} $. We set $\ord_{\cF}(0)=+\infty$ by convention.
A basis $(s_1,\ldots, s_{N_m})$ of $R_m$ is said to be \emph{compatible} with $\cF$ if  $\cF^\la R_m = {\rm span} \langle s_j \, \vert \, \ord_{\cF}(s_j) \geq \la \rangle$ for each $\la \in \R$.

\begin{example} 
The following  filtrations play an important role in this paper. 

\begin{enumerate}
\item
Given $v\in \Val_X$, there is an induced filtration $\cF_v$ of $R$  defined by 
\[\cF_v^\la R_m : = {\{ s \in R_m \, \vert \, v(s) \geq \la \}}.\]
When $A_{X,\Delta}(v)<+\infty$, $\cF_v$ is linearly bounded \cite[Lemma 3.1]{BJ20}. 
\item Similarly, any effective $\bQ$-divisor $G$ on $X$ induces a filtration $\cF_G$ of $R$ by setting
\[
\cF_G^\lambda R_m := \{s\in R_m \, \vert \, \{s=0\}\geq \lambda G\}.
\]
\item A \emph{test configuration}  $(\cX,\cL)$ of $(X,L)$ induces a finitely generated $\bZ$-filtration of $R$. See \cite[Section 2.5]{BHJ17} for details. 
\end{enumerate}
\end{example}

\subsubsection{Successive minima}
Given a basis $(s_1,\ldots, s_{N_m})$ of $R_m$ compatible with $\cF$, the numbers
\[
\la_{j}^{(m)} := \ord_{\cF}(s_j) \quad \text{ for } 1 \leq j \leq N_m.
\]
are called the \emph{successive minima} of $\cF$ along $R$. 
Since
$\frac{d}{d\la} \dim \cF^\la R_m = -\sum_{j} \delta_{\la_j^{(m)}}$
 the values $(\la_j^{(m)})_j$ are independent of the choice of compatible basis up to reordering.
We write $\la_{\max}^{(m)}= \max \{ \la \in \R \, \vert \,  \cF^\la R_m \neq 0 \}$ and set
$\la_{\max} : = \sup_{m\geq 1}  \frac{\la_{\max}^{(m)} }{mr}$.

%\subsubsection{Scaling and translation} 
%Given a filtration $F$ of $R$, the translate of $F$ by  $c\in \R_{>0}$ and scaling by $d\in \R$
%are the  filtrations $G$ and $H$ of $R$ defined by
%\[
%G^\la R_m : = F^{\la/c} R_m 
%\quad \text{ and } \quad
%H^\la R_m : = F^{\la-mrd} R_m 
%.\]
%The jumping numbers satisfy the relation
%\[
%\la_{j}^{(m)}(G) =c \la_{j}^{(m)}(F) 
%\quad \text{ and  } \quad 
%\la_{j}^{(m)}(H) = \la_j^{(m)}(F) -mr d
%.\]

\subsubsection{Graded linear series}\label{sss-gradedlinearseries}
A graded linear series  $V_\bullet= (V_m)_{m \in \N}$ of $L$ is a collection of vector subspaces $V_m \subset R_m$ satisfying $V_{m} \cdot V_{m'} \subset V_{m+m'}$.
The \emph{volume} of $V_\bullet$ is the value 
\[
\vol(V_\bullet): = \limsup_{m\to \infty} \frac{ \dim V_m}{m^n/n!}. 
\]
%We say $V_\bullet$ \emph{contains an ample series} \cite[Definition 1.1]{BC11}  \cite[Remark 2.10]{LM09}
% if \begin{enumerate}
%\item[(i)] $V_m\neq 0$ for $m\gg0$ and 
%\item[(ii)] there exists 
%$\Q$-divisors $A$ and $E$, ample and effective respectively, such that $L=A+E$ and ${H^0(X,\cO_{X}(mA)) \subset  V_m \subset H^0(X, \cO_{X}(mL))}$
%for all sufficiently divisible $m>0$. 
%\end{enumerate}
%When this condition is satisfied, $\vol(V_\bullet)$ is a limit and $>0$ by \cite[Theorem 2.13]{LM09}.

Given a filtration $\cF $ of $R$ and $s\in \R$, 
we define a graded linear series $V_\bullet^{(s)} $ by setting $V_{m}^{(s)} := \cF^{mr s } R_m$. 
When the choice of filtration is not clear from context, we will denote it by  $V_\bullet^{\cF,(s)}$. 
%The function $s\mapsto \vol(V_\bullet^{(s)})$ is non-decreasing   and   zero for  $s> \la_{\max}$. 

The following result is a consequence of \cite{BC11}. See  \cite[Theorem 5.2]{BHJ17}.

\begin{prop}\label{p:BC11}
Let $\cF$ be a linearly bounded filtration of $R$. 
\begin{enumerate}
\item[(i)] For each $s< \la_{\max}$, $\vol(V_\bullet^{(s)})$ is a limit.
\item[(ii)] The function $s\mapsto \vol(V_\bullet ^{(s)})^{1/n}$ is concave on $(-\infty, \la_{\max})$ and vanishes on $(\la_{\max}, \infty)$ (hence, it is continuous away from $s=\la_{\max}$).
\end{enumerate}
\end{prop}

\subsubsection{Base ideals}
Given a linearly bounded filtration  $\cF $ 
of $R$, we set 
\[
I_{m,\la} : = \im (  \cF^\la R_m  \otimes_k \cO_{X}(-mL) \to \cO_X),\]
which is the base ideal of the linear series $\cF^\la R_m $.
Note that $I_{m,\la} \cdot I_{m',\la'} \subset I_{m+m',\la+\la'}$.

For each $\la\in \R$, we write $I_{\bullet}^{(\la)}$ % (or $I_{\cF,\bullet}^{(\la)}$ when the choice of the filtration is not clear from the context) 
for the graded sequence of ideals on $X$ defined by setting $I_{m}^{(\la)}: = I_{m,\la mr}$.
For each $m \in \Z_{>0}$, we set 
\[
\cI_{m} :  = \sum_{i\in \Z}  I_{m,i} t^{-i} \subset K(X)(t) \simeq  K(X\times \A^1) .
\]
Note that $ t^{  \lceil \la_{\max}^{(m)} \rceil } \cdot \cI_{m} \subset \cO_{X_{ \A^1}}$, since $I_{m,i}=0$ for $i>   \la_{\max}^{(m)}$. Therefore, $\cI_m$ is a fractional ideal on $X_{ \A^1} := X\times \A^1$.

%\begin{lem} Let $\cF$ be a linearly bounded filtration of $R$. 
%\begin{enumerate}
%\item[(i)] The function $s\mapsto v(I_\bullet^{(s)})$ is non-decreasing  and convex for $s< \la_{\max}$
%\item[(ii)] The function $s\mapsto \lct(X,\Delta;I_\bullet^{(s)})^{-1} $ is convex for $s<\la_{\max}$.
%\end{enumerate}
%\end{lem}

%Note that when $s<\la_{\max}$,
%$I_{m}^{(s)} \neq 0$ for $m\gg0$ by Proposition \ref{p:BC11}(i). Hence, $v(I_\bullet^{(s)})< +\infty$.

%\begin{proof}
%For (i), the function is clearly non-decreasing, since $I_{\bullet}^{(s)} \supset I_\bullet^{(s')}$ when $s<s'$. To see the convexity, fix $s<s' < \la_{\max}$.
%Observe that for positive integers $a$ and $b$, 
%\[
%v(I_{\bullet}^{(s)}) + bv(I_{\bullet}^{(s')}) \geq (a+b)v(I_{\bullet}^{((sa+s'b)/(a+b))}),\]
%since$I_{m,ms}^a \cdot I_{m,ms'}^b \subset I_{m(a+b),m(sa+sb)}$.
%Therefore, 
%\[
%(1-t) v(I_{\bullet}^{(s)}) + tv(I_{\bullet}^{(s')}) \geq v(I_{\bullet}^{(s(1-t)a+s' t)})\]
 %for any $t\in [0,1]\cap\Q$. Using (i), the previous equality further holds for all $t\in [0,1]$. 

%To verify (ii), note that $\lct(X,\Delta;I_\bullet^{(s)})^{-1}  = \sup_{v\in \DivVal_X} \frac{v(I_\bullet^{(s)})}{A_{X,\Delta}(v)}$. Since the supremum of a collection  of convex functions is convex, the result now follows from (i).
%\end{proof}

\subsubsection{Duistermaat-Heckman Measure}
Given a linearly bounded filtration $\cF$ of $R$ and an integer $m>0$, we consider the probability measure  on $\R$ defined by
\[
\nu_m^{\cF} : = \frac{1}{N_m} \sum _{i=1}^{N_m} \delta_{(mr)^{-1} \la_{i}^{(m)}} = - \frac{ d}{d \la} \frac{ \dim \cF^{mr\la} R_m}{N_m} . 
\]
By \cite{BC11} (see  \cite[\S 5.1]{BHJ17}), $\nu_m^\cF$ converges weakly as $m\to \infty$ to the probability measure 
\[
\Dh_\cF
:=
- \frac{d}{d \la}   \frac{ \vol( V_\bullet^{ (\la)}) }{L^n} .
\]
The measure satisfies ${\rm supp}(\Dh_{\cF}) = [{\la}_{\min}, \la_{\max}]$, where 
$${\la}_{\min} := \inf \{ \la \in \R \, \vert \, \vol (V_\bullet^{(\la)} ) < (L^n)  \}$$
and $\la_{\max}$ is as defined previously.

The following statement is an extension of \cite[Lemma 5.13]{BHJ17} from divisorial valuations to valuations with finite log discrepancy.
\begin{lem}\label{l:lminFv}
If  $v\in \Val_{X}$ and $A_{X,\Delta}(v)< +\infty$, then ${\la}_{\min}(\cF_v)=0$.

\end{lem}
\begin{proof}
It is clear that $\la_{\min}(\cF_v)\geq0$, since $\cF_v^\la R =R$ for $\la \leq 0$. 
For the reverse inequality, 
fix  a log resolution $Y\to X$ of $(X,\Delta)$ and let $\xi :=c_{Y}(v)$ be the center of $v$ on $Y$. 
Note that  $ A_{Y,0}(v)$ is finite,
since $A_{X,\Delta}(v)<+\infty$ by assumption. 
An Izumi type inequality \cite[Proposition 5.10]{JM12}
implies
\[
v(f) \leq c \cdot  \ord_{\xi} (f) 
 \quad\quad \text{ for all  } f\in \cO_{Y, \xi}
 ,
 \]
 where $c:= A_{Y,0}(v)>0$, and, hence, $ \cF_v^\la \subset \cF_{c \cdot \ord_\xi}^\la R $ for all $\la\in \R$.
Therefore, %\footnote{\ZZ{It seems to me that we need the inequality and containment in the opposite direction here? Something like $v(f)\le c\cdot \ord_\xi$.\HB{Thanks for pointing this out. I had an error in the Izumi inequality and it propagated to the next line. I think it is fixed.}}},
 $\la_{\min}(\cF_v) \leq  \la_{\min}(\cF_{c \cdot  \ord_\xi}) = 0$,
 where the equality holds by \cite[Lemma 5.13]{BHJ17},  since $c\cdot \ord_\xi$ is a divisorial valuation.
\end{proof}

%\subsubsection{Relative Limit Measure} \footnote{\ZZ{It seems to me that the notations in this subsection has a lot of overlap with those in Section \ref{ss:geodesic}. Shall we combine these two small sections into one?} \CX{Yes, I agree.}\HB{OK, I consolidated it into  Section 3.1}}

\subsection{Non-Archimedean functionals}

\subsubsection{Energy functional}
Following \cite{BHJ17}, the Monge-Amp\`{e}re Energy of $\cF$ is given by 
\[
\ENA(\cF)  : =  \int_{\R} \la  \, \Dh_{\cF}( d \la) %=  \lim_{m\to \infty} \frac{1}{N_m}  \sum_{i=1}^{N_m}(mr)^{-1} \la_{i}^{(m)}, 
\]
which is the barycenter of $\Dh_{\cF}$.
When $\cF=\cF_v$ for a valuation $v\in \Val_X$ with $A_{X,\Delta}(v)<\infty$, $\ENA(\cF_v)=S(v)$ and $\la_{\max}(\cF_v) = T(v)$ where $S(\cdot)$ and $T(\cdot)$ are the expected and maximal vanishing order appearing in \cite{BJ20}. 

% For $p \in [1, \infty)$, the $L^p$-norm of $\cF$, is given by 
% \[
% \lVert \cF \rVert_p  : = \Big( \int_{\R} |\la -\overline{\la} |^p \,  \Dh_{\cF}( d\la)\Big)^{1/p}
% % =\lim_{m\to \infty} 
% % \left(
% %\frac{1}{N_m}\sum_{i=1}^{N_m} |(mr)^{-1} \la_{i}^{(m)} - \overline{\la} |^p \right)^{1/p}
% ,
% \]
% where $\overline{\la} : = \int_{\R} \la  \, \Dh_{\cF}( d \la)  = \ENA(\cF)$.
% We say a filtration is \emph{almost trivial} if $\lVert \cF \rVert_p =0$ for some (or equivalently all) $1\leq p <\infty$.

\subsubsection{Ding-functional}\label{ss:Ding}
The \emph{Ding invariant} of a linearly bounded filtration $\cF$ is defined by 
\[
\DNA(\cF) = \LNA(\cF) - \ENA(\cF),
\]
where $\LNA(\cF) = \lim_{m\to \infty} \lct(X_{\A^1},\Delta_{\A^1}, \cI_{m}^{\frac{1}{mr}} ; (t) )-1$ and 
\[ \lct(X_{\A^1},\Delta_{\A^1}, \cI_{m}^{\frac{1}{mr}} ; (t))
 := \sup \{ c \in \R \, \vert \, (X_{\A^1},\Delta_{\A^1}, \cI_{m}^{\frac{1}{mr}}\cdot (t)^c) \text{ is sub-lc} \}.\]
This  invariant was introduced in \cite{Ber16} for test configurations 
and \cites{BHJ17, Fuj-valuative} for general filtrations.

\subsubsection{$\HNA$-functional}\label{ss:H}

Following \cite{HL20-uniqueness}, for a linearly bounded filtration $\cF$ we set 
\[
\HNA(\cF) = \LNA(\cF) - \tS(\cF)
,\]
where $\LNA(\cF)$ is defined above and 
$\tS(\cF): =  - \log \int_{\R}  e^{- \la} \, \Dh_{\cF}(d \la)$.
This invariant was introduced in \cite{TZZZ13} for holomorphic vector fields
and then extended to $\R$-test configurations in \cite{DS20} and linearly bounded filtrations in \cite{HL20-uniqueness}. We set 
\[
h(X,\Delta):=\inf_\cF \HNA(\cF) 
,\]
where the infimum runs through all linearly bounded filtrations of $\cF$. By \cite[Corollary 4.7]{HL20-uniqueness}, $h(X,\Delta)\leq0$ and equality holds iff $(X,\Delta)$ is K-semistable.
%\footnote{ \HB{added. Also, maybe there is a typo in the statement of \cite[Corollary 4.7]{HL20} ``We always have $h(X)\leq 0$, with equality holds true if and only if $h(X) = 0$'' and it should be equality holds iff $X$ is K-ss }}

For a valuation $v\in \Val_X^\circ \cup \{ v_\triv \}$, we define 
\[
\tbeta_{X,\Delta}(v):= A_{X,\Delta}(v)-\tS(v),
\]
where $\tS(v)=\tS(\cF_v)$. Note that $\tbeta_{X,\Delta}(v_{\triv})=0$. By \cite[Theorem 1.5]{HL20-uniqueness}, 
\[
h(X,\Delta) = \inf_{v \in \Val_X^\circ \cup \{ v_\triv \}} \tilde{\beta}(v). 
\]
We say that $v\in \Valc\cup\{v_{\triv}\}$  computes $h(X,\Delta)$ if it achieves the above infimum.
By \cite[Theorem 4.9]{HL20-uniqueness}, there always exists a quasi-monomial valuation computing $h(X,\Delta)$.

\section{Convexity and uniqueness}\label{s:convexDing}

In this section, we will obtain the uniqueness of the valuation computing $h(X,\Delta)$. In \cite{HL20-uniqueness}, this was proved to follow from the finite generation i.e. Theorem \ref{t-finitegen}. In this section, instead of using the finite generation, we will take the approach of establishing more general convexity results. In fact, for two filtrations $\cF_0$ and $\cF_1$,
 we consider a segment in the space of filtrations $(\cF_t)_{t\in [0,1]}$, which we call the \emph{geodesic} between  the two filtrations. 
 We then introduce a probability measure on $\R^2$ that encodes 
 $\Dh_{\cF_t}$ for $t\in [0,1]$. This will allow us to deduce the convexity of a number of functionals which take the form of integrating over the Duistermaat-Heckman measure. For $\LNA$, the proof of its convexity uses the ideas from \cite{XZ20nvol} in the local setting.

Throughout, $(X,\Delta)$ is a log Fano pair, 
 $r>0$ a rational number so that $L:=-r(K_X+\Delta)$ is a Cartier divisor, and $R:=R(X,L)$.

\subsection{Geodesics and DH measures}
Fix two linearly bounded filtrations $\cF_0$ and $\cF_1$ of $R$.
For each integer $m>0$, choose a basis $(s_1,\ldots, s_{N_m})$ of $R_m$ that is compatible with both $\cF_0$ and $\cF_1$ simultaneously; see  \cite[Lemma 3.1]{AZ20} or \cite[Proposition 1.14]{BE18} for the existence of such a basis. For $1\leq i \leq N_m$, we set
\[
\la_{i}^{0,(m)} : = \ord_{\cF_0}(s_{i})  
\quad \text{ and } \quad 
\la^{1,(m)}_{i} = \ord_{\cF_1}(s_{i}).\]
The pairs $(\la_{i}^{0,(m)}, \la_{i}^{1,(m)})$  are unique up to reordering. For example, this follows from the observation that 
 $-\frac{\partial^2 }{\partial x \partial y} \dim ( \cF_0^{x}R_m \cap \cF_1^y R_m) = \sum_{i} \delta_{ \la_i^{0,(m)}, \la_{i}^{1,(m)}}$. The above basis and notation will be used in the constructions below.

 \subsubsection{Relative limit measure}
For each integer $m>0$, we define a probability measure on $\R$  by 
\[
\nu_m^{\cF_0,\cF_1} : = \frac{1}{N_m} \sum_{i=1}^{N_m} \delta_{(mr)^{-1} ( \la^{0,(m)}_{i} - \la^{1,(m)}_{i})}
.\]
It was proven in \cite{CM15} that $\nu_m^{\cF_0,\cF_1} $ converges weakly as $m\to \infty$ to a compactly supported probability measure that we denote by ${\rm RLM}_{\cF_0,\cF_1}$. See \cite[Theorem 3.3]{BJ18} for the statement and proof written in our setting.

The \emph{$L^1$-distance} \cite[Section 3.4]{BJ18} between $\cF_0$ and $\cF_1$ is defined by
\[
d_1(\cF_0,\cF_1) : =   \int_\R |\la| \RLM_{\cF_0,\cF_1}( d \la) 
%=    \lim_{m \to \infty}  \Big( \frac{1}{ N_m}  \sum_{i=1}^{N_{m}} \left| (mr)^{-1} (\la_{i,m}-\la'_{i,m}) \right|^p \Big)^{1/p}
.\]
%When $p=2$, this coincides with a notion from \cite{Cod19}.
We say $\cF_0$ and $\cF_1$ are \emph{equivalent} if $d_1(\cF_0,\cF_1)=0$.

\begin{prop}\cite[Corollary 3.13]{BJ18}\label{p:equivDh}
 If $\cF_0$ and $\cF_1$ are equivalent, then $\Dh_{\cF_0}= \Dh_{\cF_1}$. 
 \end{prop}

\subsubsection{Geodesics}\label{ss:geodesic}
Let $\cF_0  $ and $\cF_1 $ be linearly bounded filtrations of $R$.
For $t\in (0,1)$, we define a filtration $\cF_t$ of $R$ by setting
\begin{equation} \label{e:F_t defn}
\cF^{\lambda}_t R_m 
: =
 \sum_{\mu(1-t)+\nu t \ge \lambda} 
 \cF_0^{\mu} \cF_m \cap \cF_1^\nu R_m
.
\end{equation}
It is straightforward to check that $\cF_t $ is a filtration of $R$ and is linearly bounded. 
We will call $(\cF_{t})_{t\in [0,1]}$ the \emph{geodesic} connecting $\cF_0$ and $\cF_1$.

An alternative way to describe $\cF_t$  is in terms of the basis
 $(s_{1},\ldots, s_{N_m})$ of $R_m$ fixed earlier.
Indeed, since
$\cF_0^\mu R_m \cap \cF_1^\nu R_m =  {\rm span}  \langle s_{i} \, \vert \, \la^{0,(m)}_{i} \geq \mu  \text{ and } \la^{1,(m)}_{i} \geq \nu \rangle$, 
it follows that 
\[
\cF_t^\la R_m={\rm span} \langle s_{i} \, \vert \, \la^{0,(m)}_{i}(1-t) + \la^{1,(m)}_{i} t \geq \la \rangle.
\]
Therefore, the basis    $(s_{1},\ldots, s_{N_m})$  is compatible with $\cF_t$ and 
$\ord_{\cF_t}(s_{i}) = (1-t) {\la}_{i}^{0,(m)} + t {\la}_{i}^{1,(m)}$.
As a consequence of this observation, 
\begin{equation}\label{e:nuF_t}
\nu_{m}^{\cF_t} =  \frac{1}{N_m} \sum_{i=1} \delta_{(mr)^{-1} ( \la^{0,(m)}_{i}(1-t) + \la^{1,(m)}_{i} t)}   .
\end{equation}
In the following section, will analyze  a measure on $\R^2$ that encodes \eqref{e:nuF_t} for each $t\in [0,1]$. 

\begin{rem}
 In the language of graded norms,  the  definition of $(\cF_t)_{t\in [0,1]}$ appears in the work of Reboulet in a more general setting and plays a key role in his theory of geodesics in the space of non-Archimedean metrics on a line bundle \cite{Reb20}. 
\end{rem}

%Given a polarized variety $(X,L)$ over a non-Archiemdean field, 
%Reboulet recently developed a theory of geodesics in the space of non-Archimean metrics on $L$ modeled on the corresponding theory in complex geometry \cite{Reb20}. 
%In his work, the above construction appears in a more general framework and using the language of graded norms, rather than filtrations, and plays an important role in constructing geodesics between non-Archimedean metrics. 

\subsubsection{Duistermaat–Heckman measures}\label{s:measureR2}
For each $m>0$,  we consider the probability measure   on $\R^2$ defined by 
\[
\bnu_m: =  \frac{1}{N_m} \sum_{i=1}^{N_m}\delta_{ ( (mr)^{-1}\la^{0,(m)}_{i} , (mr)^{-1}\la^{1,(m)}_{i} )} = 
\frac{ \partial^2}{ \partial x \partial y }  \frac{ \dim (\cF_0^{mrx} R_m \cap \cF_1^{mry} R_m)}{N_m}
\]
Since $\cF_0$ and $\cF_1$ are assumed to be linearly bounded, we may fix $C>0$ so that $\cF_i^{Cmr}R_m = 0$ and $\cF_i^{-Cmr} R_m = R_m$ for both $i=0,1$.
% \[
% -C
% \leq 
% \inf_{m} \frac{\la_{\min}^{0,(m)}}{mr}  \leq \sup_{m} \frac{\la_{\max}^{0,(m)}}{mr} \leq C
% \quad \text{ and } \quad
% -C
% \leq 
% \inf_{m} \frac{ \la^{1,(m)}_{\min} }{mr} \leq \sup_m \frac{ \la^{1,(m)}_{\max}}{mr} \leq C
% .\]
Hence, ${\rm supp}(\bnu_m)$ is contained in the bounded set  $[-C,C]\times [-C,C]$.

%Therefore, Prokhorov's Theorem guarantees that some subsequence of $\bnu_m$ converges weakly to a probability measure on $\R^2$. 
%While the latter  is sufficient to prove the main results in this paper, we will show the following  stronger  result that is likely of independent interest.

\begin{thm}\label{t:bnuconverges}
The sequence $\bnu_m$ converges weakly as $m\to \infty$ to the compactly supported probability measure
\[
{\rm DH}_{\cF_0,\cF_1}:= \frac{ \partial^2}{ \partial x \partial y} \frac{ \vol (W_\bullet^{(x,y)})}{L^n},
\]
where $W_\bullet^{(x,y)}$ is the graded linear series defined by $W_{m}^{(x,y)} = \cF_0^{mrx} R_m \cap \cF_1^{mry} R_m$.
\end{thm}

We will call ${\rm DH}_{\cF_0,\cF_1}$ the {\it compatible DH measure} of the two filtrations.
The use of the measure is that it encodes $\Dh_{\cF_t}$ for $t\in [0,1]$, as well as $\RLM_{\cF_0,\cF_1}$ (see Proposition \ref{p:bnuproject}).

To prove Theorem \ref{t:bnuconverges}, we  analyze the   following functions $\R^2\to [0,1]$  that are non-increasing in both variables:
 \[
 f_m(x,y) =  \frac{\dim ( W_{m}^{(x,y)} )} {N_m}
 \quad \text{ and } \quad
 f(x,y) :=\limsup_{m\to \infty} f_m(x,y)= \frac{\vol(W_\bullet^{(x,y)})}{(L^n)},
 \]
 as well as the locus
 $P = \overline{\bigcup_{m\geq 1} P_m}$ where $P_{m} = \Supp(f_m)$.

\begin{prop}\label{p:Pconvex} The set $P$ is convex and $\Int(P) = \cup_m \Int(P_m)$
\end{prop}

\begin{proof}
Using property (F4) of a filtration, it follows that
\begin{equation}\label{eq:P_minclusion}
 c m P_{m}+ d q P_{q} \subset (cm+dq) P_{mc+qd} \quad \text{ for all } c,d,m,q \in \Z_{>0}
.\end{equation}
 Indeed, if $(x,y) \in     c m P_{m}$ and $(x',y') \in d q P_{q}$,
    then there exist non-zero sections 
    \[
    s \in \cF_{0}^{rx/c} R_{m} \cap \cF_1^{ry/c} R_m \quad \text{ and } \quad
    s' \in
    \cF_{0}^{rx'/d} R_{q} \cap \cF_1^{ry'/d} R_q.
    \]
    Hence, 
    \[
    s^c s'^d \in \cF_{0}^{r(x+x')}R_{mc+md} \cap \cF_1^{r(y+y')}R_{mc+md}
    \]
    which implies $(x+x',y+y') \in ( mc+qd) P_{mc+qd}$ as desired.
Now, \eqref{eq:P_minclusion} implies: if  $p,q \in \cup_m P_m$ and $t\in [0,1]\cap \Q$, then $p(1-t)+tq \in \cup_m P_m$. Therefore, the closure of $\cup_{m} P_m$  is convex.  

To show $\Int(P) = \cup_m \Int(P_m)$, first note that the inclusion $\supset$ clearly holds.  
To see $\subset$ holds, fix $(a,b) \in \Int(P)$. 
Since $\Int(P)$ is open, we may choose $\epsilon>0$ so that $(a+\epsilon, b+\epsilon) \in \Int(P)$. 
Since $P$ is the closure of $\cup_m P_m$ and $(a+\epsilon,b+\epsilon) \in P$, there exists $(x,y) \in \cup_m P_m$ so that $ a<x$ and $b<y$.  
Using that each $f_m$ is $\geq 0$ and non-increasing in both variables, the latter implies $(a,b) \in \cup_m \Int(P_m)$ as desired. 
 \end{proof}

\begin{prop}\label{p:fconverge}
On the locus $\R^2 \setminus \partial P$, 
  $f=\lim_{m\to \infty} f_m$   and  $f$ is continuous. 
\end{prop}

\begin{proof}
The statement clearly holds on  $\R^2 \setminus P$, since $f_m$ and $f$ are both zero on that locus.
It remains to verify the statement on $\Int(P)$. 

Fix $(a,b) \in \Int(P)$. 
Let  $\cG$ denote the filtration of $R$ defined by 
$$\cG^\la R_m : = \cF_0^{\la+mra}R_m \cap \cF_1^{\la+mrb} R_m,$$ 
which is linearly bounded since both $\cF_0$ and $\cF_1$ are linearly bounded. 
Let $ V_m^{\cG,(t)}$ and  $V_\bullet^{\cG,(t)}$ be defined as in Section \ref{sss-gradedlinearseries}. If we set 
\[g_m(t) =\frac{ \dim V_m^{\cG,(t)}}{ N_m}
\quad \text{ and } \quad
g(t)= \limsup_{m\to \infty }\frac{ \vol( V_\bullet^{\cG,(t)})}{(L^n)}  ,\]
then $g_m(t) = f_m(a+t,b+t)$ and $g(t) = f(a+t,b+t)$,
since 
$V_{m}^{\cG,(t)}= W_m^{(a+t,b+t)}$. 
Note that, 
for $t< \la_{\max}(\cG)$,  $g(t)=\lim_{m\to\infty} g_m(t)$ exists and $g$ is  continuous at $t$  by Proposition \ref{p:BC11}.

We claim that $\la_{\max}(\cG)>0$. 
Indeed, using that $g_m(t) = f_m(a+t,b+t)$,   we see
\[
(mr)^{-1}\la_{\max}^{(m)}(\cG)= \sup \{ t\in \R\,   \vert \,(a+t,b+t) \in P_m \}\]
Since $(a,b) \in \Int(P)$, Proposition \ref{p:Pconvex} implies there exists $m'>0$ so that $(a,b) \in \Int(P_{m'})$. Therefore, $\la_{\max}^{(m')}(\cG)>0$ and, hence, $\la_{\max}(\cG)>0$ as desired. 

Using the above claim, we see that $\lim_{m\to\infty} f_m(a,b)=f(a,b)$, and  $f(a+t, b+t)$ is continuous at $t=0$. Since $f$ is non-increasing in both variables, the latter implies that $f$ is continuous at $(a,b)$. 
\end{proof}

Theorem  \ref{t:bnuconverges} is now an easy consequence of the previous propositions.

\begin{proof}[Proof of Theorem \ref{t:bnuconverges}]
As $m\to \infty$, $f_m$ converge pointwise to $f$ away from a set of measure zero 
by Propositions \ref{p:Pconvex} and \ref{p:fconverge}.
Since $0\leq f_m \leq 1$, the  dominated converges theorem implies $f_m \to f$ in $L_{\rm loc}^{1}(\R^2)$. 
Therefore, $f_m\to f$ as distributions and, hence, $\bnu_m=\frac{\partial^2}{\partial x \partial y} f_m \to 
\frac{\partial^2}{ \partial x \, \partial y} f$ as distributions, as well. Since each distribution  $\bnu_m$ is  a measure, \cite[Theorem 2.1.9]{Hor03}  implies 
$\Dh_{\cF_0,\cF_1}:=\frac{\partial^2}{ \partial x \, \partial y} f$ is a measure
and $\bnu_m \overset{\rm weak}{\longrightarrow}\Dh_{\cF_0,\cF_1}$ as measures.  Furthermore, the measure $\Dh_{\cF_0,\cF_1}$ is a compactly supported  probability measure, since it is a weak  limit of probability measures with uniformly bounded support.
\end{proof}

\begin{prop}\label{p:bnuproject}
Fix $t\in [0,1]$, $c\in \R_{>0}$, and $d\in \R$. Consider the maps $p,q:\R^2 \to \R$ defined by $p(x,y)= (1-t)x+ty$ and  $q(x,y) = x-c(y+d)$. 
The following hold:
\begin{enumerate}
\item $p_*(\Dh_{\cF_0,\cF_1}) = \Dh_{\cF_t}$, and
\item $q_*(\Dh_{\cF_0,\cF_1}) = \RLM_{\cF_0,\cG}$, where $\cG$ is filtration given by $\cG^\la R_m:=\cF_1^{(\la-dmr)/c} R_m$.
\end{enumerate}
\end{prop}

\begin{proof}
Observe that
$p_*(\bnu_m) = \nu_m^{\cF_t} $
and 
$q_*(\bnu_m) = \bnu_m^{\cF_0,\cG}$.
Therefore,
 $p_*(\bnu_m) \overset{ {\rm weak}}{ \longrightarrow} \Dh_{\cF_t}$ and $q_*(\bnu_m)\overset{ {\rm weak}}{ \longrightarrow} \RLM_{\cF_0,\cG}$.
By Theorem \ref{t:bnuconverges} and the continuity of $p$ and $q$, we also have
 $p_*(\bnu_m) \overset{ {\rm weak}}{ \longrightarrow} p_*(\Dh_{\cF_0,\cF_1})$ and $q_*(\bnu_m)\overset{ {\rm weak}}{ \longrightarrow} q_*(\Dh_{\cF_0,\cF_1})$. Since weak limits of measure on $\R^2$ are unique, the result follows. 
\end{proof}

\subsection{Convexity}\label{ss:convex}

In this section, we prove the following result on the convexity of the non-Archimedean Ding and $H$-functionals.

\begin{thm}\label{t:DHconvex}
Let $\cF_0$ and $\cF_1$ be linearly bounded filtrations of $R$ and $(\cF_t)_{t\in [0,1]}$ be the geodesic connecting them. 
For $t\in (0,1)$, the following hold:
\begin{enumerate}[label=(\roman*)]
    \item $\DNA(\cF_t)\leq (1-t) \DNA(\cF_0) +t \DNA(\cF_1)$;
    \item $\HNA(\cF_t) \leq (1-t) \HNA(\cF_0) + t \HNA(\cF_1)$.
\end{enumerate}
Furthermore, the inequality in (ii)   is strict unless there exists $d\in \R$ so that $d_{1}(\cF_0,\cG)=0$, where $\cG$ is the filtration defined by $\cG^{\la} R_m := \cF_1^{\la-dmr}R_m$.
\end{thm}

The result is an algebraic analogue of a theorem of 
Berndtsson on the convexity of the Ding-functional along geodesics in the space of K\"ahler potentials \cite{Ber15}.
In forthcoming work, Reboulet shows that Theorem \ref{t:DHconvex} can in fact be deduced from Berndtsson's convexity result when $X$ is smooth \cite{Reb21}.
The proof below is self-contained and purely algebraic.

To prove Theorem \ref{t:DHconvex}, we first show the convexity of $\LNA$ along geodesics. 
For this, we first compare $\LNA$ with the log canonical slope of a filtration $\cF$, defined as (c.f. \cite[Definition 1.3 and Lemma 4.13]{XZ20cm})
\[
\mu(\cF) := \mu_{X,\Delta}(\cF) := \sup \left\{s\in\bR\,|\,\lct(X,\Delta;I^{(s)}_\bullet)\ge \frac{1}{r}\right\} = \sup \left\{s\in\bR\,|\,\lct(X,\Delta;I^{(s)}_\bullet)> \frac{1}{r}\right\}.
\]

\begin{lem} \label{l:L^NA=lc slope}
For any linearly bounded filtration $\cF$ of $R$ we have $\mu(\cF)=\LNA(\cF)$, and there exists some valuation $v\in\Valc\cup\{v_{\triv}\}$ such that 
\begin{equation} \label{e:v(I^(s))>=s up to shift}
    r^{-1}v(I_\bullet^{(\lambda)})\ge \lambda+A_{X,\Delta}(v)-\LNA(\cF)
\end{equation}
for all $\lambda\in\bR$. Moreover, if $\cF$ is a finitely generated $\bZ$-filtration and $\LNA(\cF)<\lambda_{\max}(\cF)$, then $v$ can be chosen to be a divisorial lc place of some $\bQ$-complement.
\end{lem}

Recall that a valuation $v$ is said to be an lc place of some $\bQ$-complement if there exists some effective $\bQ$-divisor $\Gamma\sim_\bQ -(K_X+\Delta)$ such that $(X,\Delta+\Gamma)$ is lc and $A_{X,\Delta+\Gamma}(v)=0$.

\begin{proof}
By \cite[Theorem 4.3]{XZ20cm}, we already have $\mu(\cF)\ge \LNA(\cF)$, thus it suffices to show $\mu(\cF)\le \LNA(\cF)$. By \cite[Theorem 7.3]{JM12}, $\lct(X_{\bA^1},\Delta_{\bA^1}+\cI_\bullet^{\frac{1}{r}} ; (t))=\LNA(\cF)+1$ is computed by some $\bG_m$-invariant valuation $w\in \Val_{X\times\bA^1}^\circ$ (the $\bG_m$-equivariant version is not proved in \cite{JM12}, but is not hard to achieve from the proof). By \cite{BHJ17}*{Lemma 4.2}, up to rescaling $w$ is the Gauss extension of a valuation $v\in \Valc\cup\{v_{\triv}\}$, i.e. $w(ft^i)=v(f)+i$ for any $0\neq f\in K(X)$ and  $i\in\bZ$. Since $w$ computes the lct and $w(t)=1$, we have
\[
\LNA(\cF)+1=A_{X_{\bA^1},\Delta_{\bA^1}}(w)-w(\cI_\bullet^{\frac{1}{r}})=A_{X,\Delta}(v)+1-w(\cI_\bullet^{\frac{1}{r}}).
\]
Thus, $A_{X,\Delta}(v)-\LNA(\cF)=w(\cI_\bullet^{\frac{1}{r}})\le w(\cI_m^{\frac{1}{mr}})\le\frac{v(I_{m,i})-i}{mr}$ for all $m\in\bN$ and $i\in\bZ$. It follows that 
\[
r^{-1}v(I_\bullet^{(\lambda)})\ge \lambda+A_{X,\Delta}(v)-\LNA(\cF)
\]
for all $\lambda\in\bR$. In particular, $\lct(X,\Delta;I_\bullet^{(\lambda)})<r^{-1}$ for any $\lambda>\LNA(\cF)$. By the definition of log canonical slope, this implies $\mu(\cF)\le \LNA(\cF)$ and proves the first part of the lemma.

If $\cF$ is finitely generated, then $\cI_{pm}=\cI_m^p$ for any sufficiently divisible $m,p\in\bN$ and thus $w$ can be chosen as a divisorial valuation. Since $\LNA(\cF)<\lambda_{\max}(\cF)$, the valuation $v$ cannot be the trivial one, otherwise $\eqref{e:v(I^(s))>=s up to shift}$ becomes $\lambda\le \LNA(\cF)$ for all $\lambda< \lambda_{\max}(\cF)$ and therefore $\lambda_{\max}(\cF)\le \LNA(\cF)$, a contradiction.
%\footnote{
%\HB{Minor comment: Would it be more precise/detailed to write ``$ \lambda\to \lambda_{\max}(\cF)^-$'', since taking the limit from the right doesn't yield a contradiction}\ZZ{Fixed.}}
By \cite[Lemma 4.1]{BHJ17}, we know that $v$ is divisorial. Let $\fa_m=I_{m,\mu mr}$. Again $\fa_{pm}=\fa_m^p$ for any sufficiently divisible $m,p\in\bN$ as $\cF$ is finitely generated. By \eqref{e:v(I^(s))>=s up to shift}, we have $r^{-1}v(\fa_\bullet)\ge A_{X,\Delta}(v)$, thus from the definition of log canonical slope, we see that $v$ necessarily computes 
\[
\lct(\fa_\bullet)=m\cdot \lct(\fa_m)=m\cdot \lct(X,\Delta;\{s=0\})
\]
for sufficiently divisible $m$ and general $s\in \cF^{\mu mr} R_m$. As $K_X+\Delta+\frac{1}{mr}\{s=0\}\sim_\bQ 0$, this easily implies that $v$ is an lc place of some $\bQ$-complement.
\end{proof}

\begin{rem} \label{r:T-inv}
From the above proof it is clear that if the filtration $\cF$ is $\bT$-equivariant for some torus $\bT<\Aut(X,\Delta)$, then the valuation $v$ can be chosen to be $\bT$-invariant as well.
\end{rem}

\begin{rem}
Lemma \ref{l:L^NA=lc slope} immediately implies that $\beta(\cF)=\DNA(\cF)$ (see \cite{XZ20cm}*{Definition 4.1} for the definition $\beta(\cF)$) for any linearly bounded multiplicative filtration $\cF$. 
\end{rem}

\begin{cor} \label{c:LNA<=A}
For any $v\in\Valc$ we have $\LNA(\cF_v)\le A_{X,\Delta}(v)$ and $\HNA(\cF_v)\le \tbeta_{X,\Delta}(v)$.
\end{cor}

\begin{proof}
It is not hard to see from the definition that $\mu(\cF_v)\le A_{X,\Delta}(v)$ (c.f. \cite{XZ20cm}*{Proposition 4.2}), thus the first inequality follows from Lemma \ref{l:L^NA=lc slope}. The second inequality follows from the first and the definition of $\HNA$ and $\tbeta$.
\end{proof}

Given the equality $\LNA(\cF)=\mu(\cF)$ (see Lemma \ref{l:L^NA=lc slope}), we can establish the convexity of $\LNA$ using  \cite{XZ20nvol}.

\begin{prop} \label{p:Lconvex}
Let $\cF_0$ and $\cF_1$ be linearly bounded filtrations of $R$ and $(\cF_t)_{t\in [0,1]}$ be the geodesic connecting them. 
For $t\in (0,1)$, $\LNA(\cF_t) \leq (1-t) \LNA(\cF_0) + t \LNA(\cF_1)$.
\end{prop}

\begin{proof}
It is not hard to see that the statement is unaffected by translation of the filtrations. Thus by Lemma \ref{l:L^NA=lc slope}, we may assume that after shifting the filtrations, there exists valuations $v_0, v_1$ on $X$ with $A_{X,\Delta}(v_i)<\infty$ such that $\LNA(\cF_i)=\mu(\cF_i)=A_{X,\Delta}(v_i)$ and  $v_i(I_{\cF_i,\bullet}^{(\lambda)})\ge r\lambda$ for all $\lambda\in\bR$ and $i=0,1$. In particular, $\cF_i^\lambda R\subseteq \cF_{v_i}^\lambda R$ for all $\lambda\in\bR$.

Let $(Y=\Spec(R),\Gamma)$ denote the affine cone over $(X,\Delta)$ with respect to the polarization $L=-r(K_X+\Delta)$. Let $w_i$ be the $\bG_m$-invariant valuation on $Y$ given by $w_i(s)=mr+v_i(s)$ for $s\in R_m$ (informally $w_i=r\cdot \ord_o+v_i$). Let $\fb_{t,\bullet}$ be the graded sequence of ideals on $Y$ defined by 
\[
\fb_{t,m}=\fa_\bullet((1-t)w_0)\boxplus \fa_\bullet(tw_1):=\sum_{i=0}^m \fa_{m-i}((1-t)w_0)\cap \fa_i(tw_1).
\]
In other words, $\fb_{t,m}$ is generated by those $s\in R$ with $\lfloor (1-t)w_0(s) \rfloor + \lfloor tw_1(s) \rfloor\ge m$. For any $k\in\bZ$ and any $s\in \cF_t^{k+2} R_m$, we have $(1-t)w_0(s)+tw_1(s)\ge mr+k+2$ by \eqref{e:F_t defn}. It follows that $s$ is a section of $\fb_{t,mr+k}$ (for if $a+b\ge k+2$ then $\lfloor a \rfloor+\lfloor b \rfloor\ge k$). Therefore elements in $\cF_t^{k+2} R_m$ yield sections of $\fb_{t,mr+k}$ on $Y$ for any $k\in\bZ$.
By \cite[Theorem 3.11]{XZ20nvol}, we have
\begin{align*}
    \lct(\fb_{t,\bullet}) & \le \lct(\fa_\bullet((1-t)w_0))+\lct(\fa_\bullet(tw_1)) \le (1-t)A_{Y,\Gamma}(w_0)+tA_{Y,\Gamma}(w_1) \\
    & = 1+(1-t)A_{X,\Delta}(v_0)+tA_{X,\Delta}(v_1) = 1+(1-t) \LNA(\cF_0) + t \LNA(\cF_1).
\end{align*}
Thus, for any rational $c>(1-t) \LNA(\cF_0) + t \LNA(\cF_1)$ and any $s\in \cF^{cmr+2}_t R_m$ (where $m$ is sufficiently divisible), as it yields a section of $\fb_{t,(1+c)mr}$ on $Y$, the pair $(Y,\Gamma+\frac{1}{mr}\{s=0\})$ is not lc. Using \cite[Lemma 3.1(5)]{Kol13}, it follows that the base $(X,\Delta+\frac{1}{mr}\{s=0\})$ is not lc for any $m\in\bN$ and any $s\in \cF^{cmr+2}_t R_m$. By the definition of log canonical slope, this implies that $\mu(\cF_t)\le c$. By Lemma \ref{l:L^NA=lc slope} and the fact that $c>(1-t) \LNA(\cF_0) + t \LNA(\cF_1)$ was arbitrary, the result follows. 
\end{proof}

Using the measure $\Dh_{\cF_0,\cF_1}$ constructed in Section \ref{s:measureR2}, 
we next describe the behavior of the Monge-Amp\`{e}re Energy and $\tS$ functionals along the  geodesic.

\begin{prop}\label{p:Eaffine}
For  $t\in [0,1]$, $\ENA(\cF_t) = (1-t) \ENA(\cF_0) + t \ENA(\cF_1)$.
\end{prop}

\begin{proof}
 Set $\bnu:= \Dh_{\cF_0,\cF_1}$. We compute
\begin{equation}\label{e:EF_t}
\ENA(\cF_t) 
= \int_{\R}  \la \, \Dh_{\cF_t}(\mathrm{d} \la)
=\int_{\R^2} \left( (1-t) x+ ty \right) \,  \mathrm{d} \bnu
= (1-t) \int_{\R^2}  x \, \mathrm{d}\bnu  + t \int_{\R^2} y \,  \mathrm{d} \bnu,
\end{equation}
where the second equality is by Proposition \ref{p:bnuproject}.
From this, the result follows. 
\end{proof}

\begin{prop}\label{p:Sconvex}
For $t\in (0,1)$, 
$\tS(\cF_t) \geq (1-t) \tS(\cF_0) + t \tS(\cF_1)$.
Furthermore, the inequality is strict unless there  exists $d\in \R$ so that  $d_1(\cF_0, \cG)= 0$, where $\cG$ is the filtration defined by $\cG^\la R_m : = \cG^{\la -mrd} R_m$.
\end{prop}

\begin{proof}
Set $\bnu: = \Dh_{\cF_0,\cF_1}$, $f(x,y): = e^{-x}$, and $g(x,y) := e^{-y}$.
For $t\in [0,1]$,
\[
\tS(\cF_t) 
=
 -\log 
\int_{\R^2} e^{-\la } \, \Dh_{\cF_t}(\mathrm{d}\la)
= -\log 
\int_{\R^2} e^{-(1-t)x - ty} \, \mathrm{d}\bnu
=
- \log \lVert f^{1-t} g^{t} \rVert_{1,\bnu}
,
\]
where the second equality is by Proposition \ref{p:bnuproject}.
 H\"older's Inequality implies
\begin{align*}
-\log \lVert f^{1-t} g^{t} \rVert_{1,\bnu}
\geq 
- \log \left( \lVert f \rVert_{1,\bnu}^{1-t} \lVert g \rVert_{1,\bnu}^t \right)
&=
- (1-t) \log \lVert f \rVert_{1,\bnu} - t \log \lVert g \rVert_{1,\bnu}\\
& = (1-t) \tS(\cF_0) + t \tS(\cF_1)
\end{align*}
Furthermore,  the  inequality is strict unless
(i) $f=0$ or $g=0$ $\bnu$-a.e.  or 
(ii) there exists $c >0$ so that $f- c g =0$ $\bnu$-a.e. 

Condition (i) cannot occur, since $f$ and $g$ are $>0$. 
Condition (ii) is equivalent to saying $x-y-d=0$ $\bnu$-a.e., 
where $d: = -\ln(c)$. 
Now, if we  write $\cG$ for the filtration of $R$ defined by $\cG^\la R_m : = \cF_1^{\la-mrd}R_m$,
then
\[
\lVert x-y-d \rVert_{1,\bnu}
=
\int_{\R^2}  \left| x-y-d \right| \, \mathrm{d}\bnu
=
\int_{\R} \left| \la \right| \, {\rm RLM}_{\cF_0,\cG} (\mathrm{d}\la)
=
d_1(\cF,\cG)
,
\]
where the second is by Proposition \ref{p:bnuproject}. 
Therefore,  (ii) holds iff  $d_{1}(\cF,\cG)=0$. 
\end{proof}

\begin{proof}[Proof of Theorem \ref{t:DHconvex}]
The result follows immediately from Propositions  \ref{p:Lconvex}, \ref{p:Eaffine}, and \ref{p:Sconvex}.
\end{proof}

\subsection{Uniqueness of valuations computing $h(X,\Delta)$ }

As a consequence of the convexity results in the previous section, we prove that the minimizer of $\HNA$ is unique.

\begin{thm} \label{t:Hunique}
Assume $v$ and $w$ are valuations in $\Valc\cup\{v_{\triv}\}$. If $v,w$ both compute $h(X,\Delta)$, then $v=w$. 
\end{thm}

In \cite{HL20-uniqueness}, the previous theorem was shown under the assumption that there exists a special $\R$-test configuration computing $h(X,\Delta)$. The latter assumption will be verified in Corollary \ref{c-hfg}. 

\begin{proof}
Consider the geodesic $(\cF_t)_{t\in [0,1]}$  connecting $\cF_0 : = \cF_v$ and $\cF_1 :=\cF_w$. 
For $t\in (0,1)$, 
\[
\HNA(\cF_t) \leq (1-t) \HNA(\cF_0) + t \HNA(\cF_1) \leq (1-t) \tilde{\beta}_{X,\Delta}(v) + t \tilde{\beta}_{X,\Delta}(w)
= h(X,\Delta),
\]
where first inequality is Theorem \ref{t:DHconvex} and the second Corollary \ref{c:LNA<=A}. %\footnote{\HB{I think this lemma got deleted in the rewrite. It could either go in the prelims or after it is shown that $\mu=L$}\ZZ{Added it back.}}
Since $h(X,\Delta) \leq\HNA(\cF_t)$, the first inequality cannot be strict. 
Therefore, Theorem \ref{t:DHconvex} further implies there exists $d\in \R$ so that 
 $d_{1}(\cF_0,\cG)=0$, where  $\cG^\la R_m : = \cF_1^{\la-mrd} R_m$.
 
 Next, note that $d=0$, since
 \[
0= \la_{\min}(\cF_0) = \la_{\min}(\cG) = \la_{\min}(\cF_1) +d = d, 
 \]
 where first and last inequality is by Lemma \ref{l:lminFv} and the second by Proposition \ref{p:equivDh}. 
 Therefore, $d_{1}(\cF_0,\cF_1)=0$.  By Lemma \ref{l:v=w}, we conclude $v=w$.
\end{proof}
 
\begin{lem}[{\cite{HL20-uniqueness}*{Proposition 2.27}}] \label{l:v=w} 
Assume $v$ and $w$ are valuations in $\Val_X^\circ\cup \{v_{\triv}\}$. If $\cF_v$ and $\cF_w$ are equivalent, then $v=w$.
\end{lem}

This result was first proved in \cite{HL20-uniqueness} using the machinery of non-Archimedean metrics from \cite{BJ18}. 
For the sake of completeness, we give a proof which only uses the terminology introduced in this paper. 

\begin{proof}
It is enough to show $v(f)=w(f)$ for all $f\in R_m$ and $m \in \N$. Indeed, for any $\la\in \R$, we may choose some integer $m\gg 0$ such that $ \cO_X(mL)\otimes \fa_\la (v)$ is globally generated, where $\fa_\la (v):=\{f\in \cO_X\,\vert\,v(f)\ge \la\}$ denotes the valuation ideal. If $v(f)=w(f)$ for all $f\in H^0(X, \cO_X(mL) \otimes \fa_\la (v))\subseteq R_m$, then $\fa_\la (v)\subseteq \fa_\la (w)$. Switching the role of $v$ and $w$ gives the reverse containment. Thus, $\fa_\la (v) = \fa_\la (w)$ for all $\la\in \R$, and, hence, $v=w$.

Suppose now that $a=v(f)\neq w(f)=b$ for some $f\in R_m$. Without loss of generality, we may assume $a>b$. Let $\lambda=\lambda_{\max}(\cF_w)$. Fix some $\varepsilon\in (0,a-b)$ and let $p$ be a sufficiently large integer such that $\lambda r\le (a-b-\varepsilon)p$. Consider the subspace $V_k:=f^{kp} \cdot R_k\subseteq \cF_v^{akp} R_{k(mp+1)}$. For any $g\in V_k$, we have $w(g)\le bkp+k\lambda r\le (a-\varepsilon)kp$ by our choice of $p$. Thus for any basis $(s_1,\cdots,s_N)$ of $R_{k(mp+1)}$ that is compatible with both $\cF_v$ and $\cF_w$, the part that spans $\cF_v^{akp} R_{k(mp+1)}$ contains at least $\dim V_k = \dim R_k$ elements $s_i$ with $w(s_i)\le (a-\varepsilon)kp$. In particular for these $s_i$ we have $v(s_i)-w(s_i)\ge \varepsilon kp$. It follows that 
\[
\int_\bR |\la| \rd \nu_{k(mp+1)}^{\cF_v,\cF_w}\ge \frac{\frac{\varepsilon p}{(mp+1)r}\cdot \dim R_k}{\dim R_{k(mp+1)}}
\]
for all $k\gg 0$. Letting $k\to \infty$ we get $d_1(\cF_v,\cF_w)>0$, contradicting our assumption. Thus we must have $v(f)=w(f)$ for all $m\in\bN$ and all $f\in R_m$ and therefore $v=w$.
\end{proof}

% \subsubsection{Ding functional}

% Fix $p\in (1,\infty)$ and consider the infimum 
% \[
% \inf_{\cF} \frac{ \DNA(\cF) }{ \lVert \cF \rVert_{p}}
% \]
% that runs through non-almost trivial linearly bounded filtrations of $R$. This infimum has been studied in ... 

% \begin{thm}
% If $(X,\Delta)$ is K-unstable, then any two minimizers of the above infimum are equivalent after translation and scaling. 
% \end{thm}

% \begin{proof}
% Assume $\cF_0$ and $\cF_1$ linearly bounded filtrations of $R$ that both compute the infimum. After scaling $\cF_1$ by some constant, we may assume $\DNA(\cF_0)=\DNA(\cF_1)$ and, hence, $\lVert \cF_0\rVert_p=\lVert \cF_1 \rVert_p$.
% Note that the infimum is $<0$, since $(X,\Delta)$ is K-unstable. Therefore, both Ding invariants are $<0$.

% Since $\cF_0$ and $\cF_1$ are not almost trivial,
% we may fix $t\in (0,1)$ so that $\cF_t$ is not almost trivial by Proposition \ref{p:almost2}.
%  Theorem \ref{t:DHconvex}  and Proposition \ref{p:Lpconvex} imply 
% \[
% \DNA(\cF_t) \leq (1-t)\DNA(\cF_0) + t \DNA(\cF_1) 
% = \DNA(\cF_0)
% \quad \text{ and } \quad
% \lVert \cF_t \rVert_p \leq 
%  (1-t)\lVert \cF_0\rVert_p
%  +
%  t\lVert \cF_0\rVert_p = \lVert \cF_0\rVert .\]
% These inequalities cannot be strict, since then 
% $\frac{\DNA(\cF_t)}{\lVert \cF_t \rVert_p }
% < 
% \frac{\DNA(\cF_0)}{\lVert \cF_0 \rVert_p }
% = \inf_{\cF}  \frac{\DNA(\cF)}{ \lVert \cF \rVert_p }$,
% which is not possible. 
% Therefore, Proposition implies that $\cF_0$ and $\cF_1$ are equivalent after scaling and translation. 
% \end{proof}

\begin{cor}\label{c-quasiunique}
Let $(X,\Delta)$ be a log Fano pair, there is a unique valuation computing $h(X,\Delta)$ in $\Valc\cup\{v_{\triv}\}$ and it is quasi-monomial. 
\end{cor}
\begin{proof}By \cite{HL20-uniqueness}*{Corollary 4.9}, there is a quasi-monomial valuation $v\in \Valc \cup \{v_{\triv}\}$ computing $h(X,\Delta)$. The uniqueness is by Theorem \ref{t:Hunique}.
\end{proof}

\section{Weighted stability}

In this section we provide a common ground to study the stability of both K\"ahler-Ricci solitons and triples $(X,\Delta,v_0)$ (where $v_0$ is the unique minimizer of $\tbeta$ from the previous section) in a suitably weighted sense. The results will be applied in the next section to study the finite generation property for various minimizers.

\subsection{Weighted \texorpdfstring{$\delta$}{delta}-invariants}\label{ss-weighteddelta}

We first introduce a weighted version of the stability threshold and then generalize results from \cite{BJ20} to this setting (see also \cite{RTZ20}*{Section 6}).  %\footnote{\textcolor{blue}{ZZ: I decide to switch things back to quasi-monomial setting so as to keep things shorter. It should be fine if we don't do the moment polytope case in Han-Li's YTD paper. We might also be able to delete the $\bT$ notation in $\delta$ in the end (already did). In any case the previous version has been labeled (3/5/21) in History.}}

\begin{defn}
Let $v_0\in\Val_X$ be a quasi-monomial valuation and $g\colon \bR\to \bR_+$ a continuous function. A \emph{$g$-weighted $(m,v_0)$-basis type divisor} is a divisor of the form
\[
D=\frac{1}{mrQ_m} \sum_{i=1}^{N_m} g\left(\frac{v_0 (s_i)}{mr}\right)\cdot  \{s_i=0\}
\]
where $(s_1, \cdots, s_{N_m})$ is a basis of $R_m$ that is compatible with $v_0$ and $Q_m=\sum_{i=1}^{N_m} g\left(\frac{v_0(s_i)}{mr}\right)$. Note that $D\in |-K_X-\Delta|_{\bR}$. We say  $D$ is compatible with a filtration $\cF$ on $R$ if the basis $(s_1,\cdots,s_{N_m})$ is compatible with $\cF$. In particular, we say  $D$ is compatible with a valuation $v\in\Val_X^\circ$ (resp. an effective $\bQ$-divisor $G$ on $X$) if it is compatible with the induced filtration $\cF_v$ (resp. $\cF_G$).
\end{defn}

For any linearly bounded filtration $\cF$ on $R$, let $\bnu_m$ and $\bnu=\Dh_{\cF_{v_0},\cF}=\lim_{m\to\infty} \bnu_m$ denote the measures on $\bR^2$ associated to $\cF_{v_0}$ and $\cF$ as constructed in Section \ref{s:measureR2}. We set 
\[
S_{g,m}(v_0;\cF):=\frac{\int_{\bR^2} g(x)y \rd \bnu_m}{\int_{\bR^2} g(x) \rd \bnu_m} \quad \text{and}\quad S_g(v_0;\cF):=\lim_{m\to\infty} S_{g,m}(v_0;\cF)=\frac{\int_{\bR^2} g(x)y \rd \bnu}{\int_{\bR^2} g(x) \rd \bnu}.
\]
It is clear that
\[
S_{g,m}(v_0;\cF) = \frac{1}{mrQ_m}\sum_{i=1}^{N_m} g\left(\frac{v_0 (s_i)}{mr}\right)\cdot  \ord_\cF(s_i)
\]
for any basis $(s_1,\cdots,s_{N_m})$ that is compatible with both $v_0$ and $\cF$.

\begin{lem}\label{l:S=lim S_m}
For any $\varepsilon>0$ there exists a positive integer $m_0=m_0(\varepsilon)$ such that 
\[
S_{g,m}(v_0;\cF)\le (1+\varepsilon)S_g(v_0;\cF)
\]
for any linearly bounded filtration $\cF$ of $R$ satisfying $\cF^0 R=R$ and any $m\geq m_0$.
\end{lem}

This follows from essentially the same argument as \cite[Corollary 2.10]{BJ20} except we need to use an Okounkov body that is induced by a suitable valuation.

\begin{defn}
Let $w\colon K(X)^\times \to \bZ^n$ be a valuation with values in the group $\bZ^n$ (equipped with some total ordering). Following \cite{KK}, we say that $w$ is faithful if its image equals $\bZ^n$; we say that $w$ has one-dimensional leaves if $\dim \hat{V}_\alpha \le 1$ for every $\alpha\in\bZ^n$, where $\hat{V}_\alpha:=\{f\in K(X)\,\vert\,w(f)\ge \alpha\}/\{f\in K(X)\,\vert\,w(f)> \alpha\}$. Finally we say that $w$ is a good valuation if it is faithful, has one dimensional leaves, and $n=\dim X$.
\end{defn}

\begin{lem} \label{l:goodval}
Let $v_0$ be a quasi-monomial valuation on $X$. Then there exists a good valuation $w_0\colon K(X)^\times \to \bZ^n$ and $u_0\in\bR^n_{\ge 0}$ such that 
\[
v_0(f)=\langle u_0,w_0(f) \rangle \quad \quad \text{ for all }   f\in K(X)^\times.
\]
\end{lem}

\begin{proof}
Let $r$ be the rational rank of $v_0$. Since $v_0$ is quasi-monomial, there exists a log resolution $\pi\colon Y\to X$, a regular system of parameters $y_1,\ldots, y_r$ at a point $\eta \in Y$, and $\alpha\in \bR^r_+$ such that $v_0=v_\alpha$. Let $W=C_Y(v_0)$. Possibly after further blowups, we may choose a flag $W_{\bullet}\colon W=W_0\supseteq \cdots \supseteq W_{n-r}=\{\text{point}\}$ of smooth subvarieties such that each $W_{i+1}$ is a divisor in $W_i$. Let $\nu=\nu_{W_{\bullet}}\colon K(W)^\times \to \bZ^{n-r}$ be the induced valuation as in \cite{LM09}. Any nonzero $f\in \cO_{Y,W}$ can be written as $f=c_\beta y^\beta +f_1$ for some (uniquely determined) $\beta\in\bN^r$ and  $c_\beta,f_1\in \cO_{Y,W}$ such that $\langle \alpha,\beta\rangle =v_0(f)$ and $v_0(f_1)>v_0(f)$. Moreover, $0\neq \bar{c}_\beta \in K(W)$ is well-defined and does not depend on the choice of $c_\beta$. Now consider the valuation $w_0\colon K(X)^\times \to \bZ^n$ given by setting $w_0(f)=(\beta,\nu(\bar{c}_\beta))$ for $f\in \cO_{Y,W}$. It is not hard to check from the construction that $w_0$ is faithful and has one-dimensional leaves. Clearly $v_0(f)=\langle (\alpha,0,\cdots,0),w_0(f)\rangle$. Thus $w_0$ is the good valuation we want.
\end{proof}

\begin{proof}[Proof of Lemma \ref{l:S=lim S_m}]
Since the argument is very similar to those in \cite{BJ20}, we only sketch the proof. By Lemma \ref{l:goodval}, there exists some good valuation $w_0\colon K(X)\to \bZ^n$ and some $u_0\in\bR^n_{\ge 0}$ such that $v_0(f)=\langle u_0,w_0(f) \rangle$. Let $\Sigma\subseteq\bR^n$ be the corresponding Okounkov body (see \cite{KK}), i.e. the closed convex hull of $\bigcup_{m\geq 1}\{\frac{w_0(s)}{mr}\,\vert\,s\in R_m\setminus\{0\}\}$. We regard the function $g$ also as a positive function on $\Sigma$ by $g(\alpha)=g(\langle u_0,\alpha\rangle)$. Let $\rho$ denote the Lebesgue measure on $\Sigma$ and $\rho_m$ the atomic probability measure supported on $\Sigma \cap \frac{1}{m} \Z^n$ as defined in \cite[Section 2.2]{BJ20}. Note that $\lim_{m\to\infty}\rho_m=\rho$ in the weak topology of measures (see \cite[Theorem 2.1]{BJ20}).
Using the argument of \cite[Lemma 2.2]{BJ20} and the uniform continuity of $g$ on $\Sigma$, we see that for each $\varepsilon>0$ there exists $m_0=m_0(\varepsilon)$ such that
\[
\int_\Sigma fg \rd \rho_m \le \int_\Sigma fg \rd \rho + \varepsilon
\]
for every $m\ge m_0$ and every concave function $f\colon \Sigma\to \bR$ satisfying $0\le f\le 1$. We may then apply the proof of \cite[Corollary 2.10]{BJ20} to the concave transform of $\cF$ and conclude that (after possibly enlarging $m_0$) $S_{g,m}(v_0;\cF)\le (1+\varepsilon)S_g(v_0;\cF)$ for all linearly bounded filtrations $\cF$ and all $m\ge m_0$.
\end{proof}

When $\cF$ is the filtration induced by some valuation $v\in \Valc$ (resp. effective divisor $G\neq 0$ on $X$), we will simply write $S_g(v_0;v)$ (resp. $S_g(v_0;G)$) for $S_{g}(v_0,\cF)$. % Suppose that $\bT=\bG_m^s<\Aut(X,\Delta)$ is a torus subgroup of the automorphism group (we allow $\bT=\{1\}$) and $v_0$ extends to a $\bT$-invariant good valuation, 
Let $\bT=\bG_m^s<\Aut(X,\Delta)$ be a torus subgroup of the automorphism group (we allow $\bT=\{1\}$). For any quasi-monomial valuation $v_0\in\ValT$ we set
\[
\delta_{g,\bT} (X,\Delta,v_0): = \inf_{v\in \ValT} \frac{A_{X,\Delta}(v)}{S_g(v_0;v)}.
\]
We say that $v\in\ValT$ computes $\delta_{g,\bT}(X,\Delta,v_0)$ if it achieves the above infimum. For each positive integer $m$, we also set 
\[
\delta_{g,\bT,m}(X,\Delta,v_0): = \min \{  \lct(X,\Delta;D)\, \vert \, 
D \text{ is a $\bT$-invariant $g$-weighted $(m,v_0)$-basis type divisor}\}.
\]
When $\bT=\{1\}$, we will suppress the subscript $\bT$ and write $\delta_g(X,\Delta,v_0)$ and $\delta_{g,m}(X,\Delta,v_0)$.
 % We will suppress the subscript $\bT$ when it's trivial. Note that $\delta(X,\Delta,\xi)=\delta_{g,\bT}(X,\Delta,\wt_\xi)$ where $g(x)=e^{-x}$.

% \begin{rem}
% It is not hard to verify that $\delta_{g,\bT,m}(X,\Delta,v_0)=\delta_{g,m}(X,\Delta,v_0)$ for all $m\in\bN$. We will not need this fact in our discussion.
% \end{rem}

\begin{lem} \label{l:delta=lim delta_m}
In the above setup, we have 
\[
\delta_{g,\bT}(X,\Delta,v_0)=\lim_{m\to\infty}\delta_{g,\bT,m}(X,\Delta,v_0).
\]
\end{lem}

\begin{proof}
% First note that for any $v\in \Valc$, we have $S_{g,m}(v_0;v)=v(D)$ for any $g$-weighted $(m,v_0)$-basis type divisor $D$ that's also compatible with $v$. Moreover, such a divisor $D$ can be chosen to be $\bT$-invariant by choosing compatible basis in each piece of the weight decomposition of the torus action. 
% It follows that $\delta_{g,m}(X,\Delta,v_0)\le \frac{A_{X,\Delta}(v)}{S_{g,m}(v_0;v)}$. On the other hand, it is not hard to see from \eqref{e:S_m as v(basis type)} that for any $\bT$-invariant divisor $D$, the log canonical threshold $\lct(X,\Delta;D)$ is computed by some $v\in \ValT$. It follows that
It is not hard to check from the definition that
\begin{equation} % \label{e:S_m as v(basis type)}
    S_{g,m}(v_0;v) =
 \max \{ v(D) \, \vert \, D \text{ is a $g$-weighted $(m,v_0)$-basis type divisor} \}
\end{equation}
and in fact $S_{g,m}(v_0;v)=v(D)$ for any $g$-weighted $(m,v_0)$-basis type divisor $D$ that's also compatible with $v$. Moreover, when $v_0,v\in\ValT$, such a divisor $D$ can be chosen to be $\bT$-invariant by choosing compatible basis in each component of the weight decomposition under the torus action.
Hence, 
\begin{equation} % \label{e:delta_g,m as inf}
    \delta_{g,\bT,m}(X,\Delta,v_0)
    =\inf_{v\in\ValT}
    \frac{A_{X,\Delta}(v)}{S_{g,m}(v_0;v)}.
\end{equation}
Combined with Lemma \ref{l:S=lim S_m}, the argument in the proof of \cite[Theorem 4.4]{BJ20} then yields $\delta_{g,\bT}(X,\Delta,v_0)=\lim_{m\to\infty}\delta_{g,\bT,m}(X,\Delta,v_0)$.
\end{proof}

\subsection{Reduced uniform stability}\label{ss-reduceduniformK}

In this section, we define stability notions for triples $(X,\Delta,v_0)$ where $(X,\Delta)$ is a log Fano pair, $v_0$ is a quasi-monomial valuation on $X$, and $g\colon \bR\to \bR_+$ is a continuous function.

We first define the weighted version of the non-Archimedean functional. For any linearly bounded filtration on $R$, we set
\begin{align*}
    \DNA_g(\cF) & := \LNA(\cF) - S_g(v_0;\cF),\\
    \JNA_g(\cF) & := \lambda_{\max}(\cF)-S_g(v_0;\cF). 
\end{align*}
Note that $\JNA_g(\cF)\geq 0$. We also set $\ENA_g(\cF):=S_g(v_0;\cF)$. If $(\cX,\Delta_{\cX};\cL)$ is a normal test configuration of $(X,\Delta)$, then we set $\DNA_g(\cX,\Delta_{\cX};\cL):= \DNA_g (\cF_{(\cX,\Delta;\cL)})$ and $\JNA_g(\cX,\Delta;\cL):= \JNA_g (\cF_{(\cX,\Delta_{\cX};\cL)})$. Denote by $\Aut(X,\Delta,v_0)$ the subgroup of $\Aut(X,\Delta)$ that leaves the valuation $v_0$ invariant and let $\bT<\Aut(X,\Delta,v_0)$ be a torus subgroup.

\begin{defn} 
We say that the triple $(X,\Delta,v_0)$ is $\bT$-equivariantly $g$-Ding semistable (or simply $g$-Ding semistable when $\bT=\{1\}$) if $\DNA_g(\cX,\Delta;\cL)\ge 0$ for all $\bT$-equivariant normal test configurations $(\cX,\Delta;\cL)$ of $(X,\Delta)$. 
% The triple $(X,\Delta,v_0)$ is said to be $g$-Ding polystable if it is $g$-Ding semistable and $\DNA_g(\cX,\Delta_{\cX};\cL)= 0$ for a weakly special test configuration $(\cX,\Delta_{\cX};\cL)$ only if it is a product test configuration.
\end{defn}

Denote by $M:=\Hom(\bT,\bG_m)$ the weight lattice and $N:=\Hom(\bG_m,\bT)$ the co-weight lattice. Then there is a weight decomposition $R_m=\oplus_{\alpha\in M} R_{m,\alpha}$. Recall that for any $\bT$-equivariant filtration $\cF$ and each $\eta\in N_{\bR}$, there is an $\eta$-twist $\cF_{\eta}$ of $\cF$ given by 
\[
\cF_{\eta}^\lambda R_m = \bigoplus_{\alpha\in M} \cF^{\lambda-\langle\alpha, \eta\rangle} R\cap R_{m,\alpha}.
\]
Set $\Fut_g (\eta) := \ENA_g (\cF) - \ENA_g (\cF_\eta)$. It is not hard to see from the definition that $\Fut_g$ does not depend on the choice of the filtration $\cF$ and is linear on $N_\bR$. We define the \emph{reduced $\JNA_g$-norm} of $\cF$ as 
\[
\JNA_{g, \bT}(\cF):=\inf_{\eta\in N_{\bR}} \JNA_g (\cF_\eta).
\]

\begin{defn}
We say a triple $(X,\Delta,v_0)$ is \emph{reduced uniformly $g$-Ding stable} if there exists a maximal torus $\bT<\Aut(X,\Delta,v_0)$ and some $\varepsilon>0$, such that 
\[
\DNA_g(\cX,\Delta_{\cX};\cL)\geq \varepsilon \JNA_{g,\bT} (\cX,\Delta_{\cX};\cL)
\]
for all $\bT$-equivariant normal test configurations $(\cX,\Delta_{\cX};\cL)$ of $(X,\Delta)$.
\end{defn}

Note that the above definition is independent of the choice of $\bT$ since any two maximal tori are conjugate.

\begin{lem} \label{l:Fut=0}
Let $\bT<\Aut(X,\Delta,v_0)$ be a torus. Assume that $\DNA_g\ge 0$ for any product test configurations that is induced by a one parameter subgroup of $\bT$. Then $\Fut_g\equiv 0$ on $N_\bR$.
\end{lem}

\begin{proof}
Let $\cF$ be the trivial filtration of $R$. By Lemma \ref{l:L^NA=lc slope} and \cite[Lemma A.6]{XZ20cm}, we have $\LNA(\cF)=\LNA(\cF_\eta)$ for any $\eta\in N_\bR$. Thus, $\DNA_g(\cF_\eta)=\DNA_g(\cF)+\Fut_g(\eta)=\Fut_g(\eta)$. By assumption, $\DNA_g(\cF_\eta)\ge 0$ for any $\eta\in N$. By linearity, $\Fut_g\equiv 0$ on $N$ and, hence, the same holds on $N_\bR$.
\end{proof}

We are most interested in the case $v_0=\wt_\xi$ for some torus $\bT<\Aut(X,\Delta)$ and some $\xi\in N_\bR$. In this case, we write $(X,\Delta,\xi)$ instead of $(X,\Delta,\wt_\xi)$. We note that while $\bT$ is not explicitly written out in the notion $(X,\Delta,\xi)$, it is indeed part of the data.

\begin{thm}[\cite{HL20-YTD}]\label{t-HL-YTD-soliton}
Let $(X,\Delta,\xi)$ be a triple over $\bk=\bC$. Then it admits a K\"ahler-Ricci $g$-soliton if and only if it is reduced uniformly $g$-Ding stable.
\end{thm}

\begin{defn}
Let $(X,\Delta,\xi)$ be a triple.  We say that $(X,\Delta,\xi)$ is \emph{$g$-Ding semistable} if it is $\bT$-equivariantly $g$-Ding semistable. We say that $(X,\Delta,\xi)$ is \emph{$g$-Ding polystable} if it is $g$-Ding semistable and $\DNA_g(\cX,\Delta_{\cX};\cL)= 0$ for a weakly special $\bT$-equivariant test configuration $(\cX,\Delta_{\cX};\cL)$ only if it is a product test configuration. 

We say that $(X,\Delta,\xi)$ is \emph{K-semistable} (resp. \emph{K-polystable}, or \emph{reduced uniformly Ding stable}) if it is $g$-Ding semistable (resp. $g$-Ding polystable, or reduced uniformly $g$-Ding stable) for $g(x)=e^{-x}$.
\end{defn}

\begin{rem}The above definition of K-polystability agrees with the notion in \cite{BWN14} when $\T$ is the torus of smallest dimension such that $\xi\in N_{\R}$;
see \cite[Remark 2.47]{HL20-uniqueness}.
Note that K-polystability of K\"ahler-Ricci solitons is proved in \cite[Theorem 1.5]{BWN14}. Later we will see that the definition indeed does not depend on the choice of $\T$ (see Remark \ref{rem-equiv}).
\end{rem}

The following is the main result of this subsection.

\begin{lem} \label{l:wstc}
Let $v_0$ be a quasi-monomial valuation on $X$ and let $\bT<\Aut(X,\Delta,v_0)$ be a torus. Let $g\colon \bR\to \bR_+$ be a continuous function and $c\in [0,1)$. Then the following are equivalent:
\begin{enumerate}
    \item $\DNA_g(\cX,\Delta_{\cX};\cL)\geq c\cdot \JNA_{g,\bT} (\cX,\Delta_{\cX};\cL)$ for any $\bT$-equivariant normal test configuration $(\cX,\Delta_{\cX};\cL)$ of $(X,\Delta)$.
    \item $\DNA_g(\cX,\Delta_{\cX};\cL)\geq c\cdot \JNA_{g,\bT} (\cX,\Delta_{\cX};\cL)$ for any $\bT$-equivariant weakly special test configuration $(\cX,\Delta_{\cX};\cL)$ of $(X,\Delta)$.
    \item $\DNA_g(\cF_v) \ge c \cdot \JNA_{g,\bT}(\cF_v)$ for any divisorial valuation $v\in\ValT$ that is an lc place of a $\bQ$-complement.
\end{enumerate}
\end{lem}

\begin{proof}
When $v_0=\wt_\xi$ this is treated in \cite{HL20-YTD}*{Section 7} using \cite{LX14}. Here we present a proof that is independent of \cite{LX14}. It is clear that (1) implies (2). By \cite{BLX19}*{Theorem A.2}, (2) implies (3). %\footnote{\HB{In Lemma 4.10 (2), do we want weakly special test configurations with $\cX_0$ is irreducible. Of course, both statements are true. }}
 Thus, it remains to show (3) implies (1).
To see this, let $\cF$ be a finitely generated $\bT$-equivariant $\bZ$-filtration of $R$. If $\LNA(\cF)\ge \lambda_{\max}(\cF)$, then we already have $\DNA_g(\cF)\ge \JNA_g(\cF)\ge \JNA_{g,\bT}(\cF)$. Thus to prove (1) we may assume that $\LNA(\cF)<\lambda_{\max}(\cF)$. % such that $\DNA_g(\cF) < c\cdot \JNA_{g,\bT}(\cF)$. It suffices to show that $\DNA_g(\cF_v) < c \cdot \JNA_{g,\bT}(\cF_v)$ for some divisorial lc place of complement $v\in\ValT$. Observe that Since $c\in[0,1)$ and $\ENA_g(\cF)<\lambda_{\max}(\cF)$, Lemma \ref{l:L^NA=lc slope} implies that $\mu(\cF)=\LNA(\cF)<\lambda_{\max}(\cF)$. 
By the second part of Lemma \ref{l:L^NA=lc slope} and Remark \ref{r:T-inv}, there exists some divisorial lc place of a $\bQ$-complement $v\in\ValT$ such that 
\[
r^{-1}v(I^{(\lambda)}_\bullet) - A_{X,\Delta}(v)\ge \lambda-\LNA(\cF)
\]
for all $\lambda\in\bR$. Since $\DNA_g(\cF)$ and $\JNA_{g,\bT}(\cF)$ are both translation invariant, we may shift $\cF$ so that $\LNA(\cF)=A_{X,\Delta}(v)$. The above inequality then becomes $v(I^{(\lambda)}_\bullet)\ge \lambda r$ and therefore $\cF^\lambda R\subseteq \cF_v^\lambda R$ for all $\lambda\in \bR$. 

Let $\eta\in N_\bR$ and let $\cG$ be the $\eta$-twist of $\cF_v$. Then we also have $\cF_\eta^\lambda R\subseteq \cG^\lambda R$ for all $\lambda$ and this clearly implies
\[
\ENA_g(\cF_\eta)\le \ENA_g(\cG)\quad \text{and}\quad \lambda_{\max}(\cF_\eta)\le \lambda_{\max}(\cG).
\]
Since $\LNA(\cF)=A_{X,\Delta}(v)\ge \LNA(\cF_v)$ by Corollary \ref{c:LNA<=A}, we also get $\LNA(\cF_\eta)\ge \LNA(\cG)$ by Lemma \ref{l:L^NA=lc slope} and \cite[Lemma A.6]{XZ20cm}. Note that by Lemma \ref{l:Fut=0}, \cite{BLX19}*{Theorem A.2} and (3) we have $\Fut_g\equiv 0$ on $N_\bR$, thus $\DNA_g(\cF_\eta)=\DNA_g(\cF)$ and $\DNA_g(\cG)=\DNA_g(\cF_v)$. 

Since 
\begin{equation*} % \label{e:D-cJ}
    \DNA_g(\cF)-c\cdot \JNA_g(\cF) = \LNA(\cF) - (1-c)\ENA_g(\cF)-c\cdot \lambda_{\max}(\cF),
\end{equation*}
we then obtain
\begin{align*}
    \DNA_g(\cF) - c\cdot \JNA_{g,\bT}(\cF) & \ge 
    \DNA_g(\cF_\eta) - c \cdot \JNA_g(\cF_\eta) \\
    & \ge \DNA_g(\cG) - c \cdot \JNA_g(\cG)
    = \DNA_g(\cF_v) - c \cdot \JNA_g(\cG).
\end{align*}
As $\eta\in N_\bR$ is arbitrary, this gives 
\[
\DNA_g(\cF) - c \cdot \JNA_{g,\bT}(\cF) \ge \DNA_g(\cF_v) - c \cdot \JNA_{g,\bT}(\cF_v).
\]
By (3), the right hand side is $\ge 0$. Thus, the same is true for the left hand side. Since $\cF$ is arbitrary, this proves that (3) implies (1). 
\end{proof}

We get the following generalized version of Fujita-Li valuative criterion \cite{Fuj-valuative, Li-valuative} which treat the case $g=1$ (see also
\cite{HL20-YTD}*{Theorem 5.18}).

\begin{cor} \label{c:ss=delta>=1}
A triple $(X,\Delta,v_0)$ is $\bT$-equivariantly $g$-Ding semistable if and only if $\delta_{g,\bT}(X,\Delta,v_0)\ge 1$.
\end{cor}

\begin{proof}
% If $(X,\Delta,v_0)$ is $g$-Ding semistable, then $\DNA_g(\cF) \geq 0$ for all linearly bounded filtrations $\cF$ of $R$. Indeed, this follows from the proof of \cite[Theorem 2.29]{Li19}, where the result is shown when $v_0$ is the trivial valuation and  $g=1$.
% \footnote{\HB{Added. I didn't see this statement (checking $g$-Ding semistability using non-finitely generated filtrations) was mentioned earlier. Feel free to add more details... I am glossing over the part of showing $S(v_0, \cF) = \lim_m S(v_0, \cF_m)$.
% }}
% Therefore, Corollary \ref{c:LNA<=A} gives  $A_{X,\Delta}(v) - S_g(v_0,v)\ge \DNA_{g}(\cF) \geq 0$ for any $v\in\Valc$.  Thus, $\delta_g(X,\Delta,v_0)\ge 1$. 
Assume that $(X,\Delta,v_0)$ is $\bT$-equivariantly $g$-Ding semistable. If $\delta:=\delta_{g,\bT}(X,\Delta,v_0) < 1$, then we may choose some $\varepsilon>0$ such that $(1+\varepsilon)^2 \delta<1$. Let $m\gg 0$ be such that $\delta_m:=\delta_{g,\bT,m}(X,\Delta,v_0)<(1+\varepsilon)\delta$. By Lemma \ref{l:S=lim S_m}, we may also assume that $S_{g,m}(v_0;v)\le (1+\varepsilon)S_g(v_0;v)$ for any $v\in\Valc$. By definition, $\delta_m=\lct(X,\Delta;D)$ for some $\bT$-invariant $g$-weighted $(m,v_0)$-basis type divisor $D$. Thus if $w$ is a $\bT$-invariant divisorial valuation that computes the lct, we would have $$A_{X,\Delta}(w)=\delta_m S_{g,m}(v_0;w)<(1+\varepsilon)^2 \delta\cdot S_g(v_0;w)<S_g(v_0;w).$$ On the other hand, since $\delta_m<1$, we know that $-(K_X+\Delta+\delta_m D)$ is ample, thus $\gr_w R$ is finitely generated by \cite{BCHM}*{Corollary 1.4.3}. In other words, $w$ induces a $\bT$-equivariant test configuration of $(X,\Delta)$. As $(X,\Delta,v_0)$ is $\bT$-equivariantly $g$-Ding semistable, Corollary \ref{c:LNA<=A} gives $A_{X,\Delta}(w) - S_g(v_0;w)\ge \DNA_{g}(\cF_w) \geq 0$, a contradiction. Therefore, $\delta_{g,\bT}(X,\Delta,v_0)\ge 1$ when $(X,\Delta,v_0)$ is $\bT$-equivariantly $g$-Ding semistable.

Conversely, if $\delta_{g,\bT}(X,\Delta,v_0)\ge 1$, then for any $v\in\ValT$ we have $A_{X,\Delta}(v)\ge S_g(v_0;v)$. If $v$ is an lc place of a $\bQ$-complement, then we also have $\LNA(\cF_v)=A_{X,\Delta}(v)$ (c.f. \cite{XZ20cm}*{Proposition 4.2}). Thus $\DNA_g(\cF_v)\ge 0$ for any $v\in\ValT$ that is an lc place of a $\bQ$-complement. By Lemma \ref{l:wstc} (with $c=0$), this implies that $(X,\Delta,v_0)$ is $\bT$-equivariantly $g$-Ding semistable.
\end{proof}

% Since $\cF$ is $\bT$-equivariant, we may choose $v$ to be $\bT$-invariant by the proof of Lemma \ref{l:L^NA=lc slope}. This implies that for any $s\in R_m$ we have $\frac{v(s)}{mr} - A_{X,\Delta}(v) \ge \frac{\ord_\cF(s)}{mr} - \LNA(\cF)$, or
% \[
% \LNA(\cF)-\frac{\ord_\cF(s)}{mr} \ge A_{X,\Delta}(v)-\frac{v(s)}{mr} \ge \LNA(\cF_v)-\frac{v(s)}{mr}.
% \]

% Summing over a $\bT$-invariant basis $(s_i)$ of $R_m$ and letting $m\to \infty$, it then follows immediately from the definition of $\ENA_g(\cF)$ that 
% \[
% \LNA(\cF)-\ENA_g(\cF_\eta)\ge \LNA(\cF_v)-\ENA_g((\cF_v)_\eta)
% \]
% for any $\eta\in N_\bR$. Similarly we have 
% \[
% \LNA(\cF)-\lambda_{\max}(\cF_\eta)\ge \LNA(\cF_v)-\lambda_{\max}((\cF_v)_\eta).
% \]
% Recall $\DNA_g(\cF_\eta)=\DNA_g(\cF)$. Combined with \eqref{e:D-cJ} and \cite[Lemma A.6]{XZ20cm} this gives
% \[
% \DNA_g(\cF) - c \cdot \JNA_g(\cF_\eta) \ge \DNA_g(\cF_v) - c \cdot \JNA_g((\cF_v)_\eta)
% \]
% for any $c\in[0,1)$ and any $\eta\in N_\bR$. It follows that

\subsection{Existence of minimizer}\label{ss-reducedminimizer}

Throughout this section, let $\bT<\Aut(X,\Delta)$ be a torus and let $\xi\in N_\bR$. We aim to prove the following statement.%\footnote{\YL{Do we want to include a version for $\delta_g\leq 1$?} \ZZ{Right now I don't know how to prove minimizers exist. Without the $\bT$-action and assuming field is uncountable, we can probably just follow \cite{BJ20}. The general case seems harder. Any suggestions?}}

\begin{prop} \label{p:red uks minimizer}
Assume that $(X,\Delta,\xi)$ is $g$-Ding semistable but not reduced uniformly $g$-Ding stable. Then there exists a $\bT$-invariant quasi-monomial valuation $v$ that is not of the form $\wt_\eta$ for any $\eta\in N_\bR$ such that $1=\delta_g (X,\Delta,\xi)(:=\delta_{g,\bT} (X,\Delta,\wt_{\xi}))$
is computed by $v$.
\end{prop}

\begin{proof}
% If $\bT$ is not maximal, then we can choose $v$ to be induced by any vector in $N(\bT^{\rm max})\setminus N$, for a maximal torus $\bT^{\rm max}$ containing $\bT$. Thus we may assume $\bT$ is maximal. \footnote{\CX{I add this paragraph.}}

The argument is very similar to those in \cite{XZ20cm}*{Appendix}, so we only give a sketch. First note that the assumption remains true if we enlarge the torus $\bT$, and clearly if the conclusion holds for a maximal torus containing $\bT$, then it also holds for $\bT$. Thus we may assume that $\bT$ is a maximal torus.

By Lemma \ref{l:wstc} and \cite[Theorem 3.5]{BLX19}, we know that there exist some integer $N>0$ and a sequence of divisorial lc places of $N$-complements $v_i\in \ValT$ (not of the form $\wt_\eta$) such that 
\[
\lim_{i\to\infty}\frac{\DNA_g(\cF_{v_i})}{\JNA_{g,\bT}(\cF_{v_i})} = 0.
\]
By the constructibility result \cite{XZ20cm}*{Lemma A.11} and arguing as in the proof of \cite{XZ20cm}*{Theorem A.5}, we may assume that the $v_i$'s are lc places of the same $\bQ$-complement and after rescaling $v=\lim_i v_i\in \ValT$ exists and $v\neq \wt_\eta$ for any $\eta\in N_\bR$. By the following Lemma \ref{l:S_g continuous}, we have $\lim_i S_g(v_0;v_i) = S_g(v_0;v)$. Since $v_i$ are lc places of $\bQ$-complements, it follows from Lemma \ref{l:L^NA=lc slope} that $\LNA(\cF_{v_i})=A_{X,\Delta}(v_i)$ and thus $\lim_i \DNA_g(\cF_{v_i}) = \DNA_g(\cF_{v})$. Since the function $v\mapsto \lambda_{\max}(\cF_v)=T_{X,\Delta}(v)$ is also continuous on $\QM(Y,E)$ by \cite[Proposition 2.4]{BLX19}, we get $\lim_i \JNA_g(\cF_{v_i}) = \JNA_g(\cF_{v})$ as well. Thus, 
\[
\frac{\DNA_g(\cF_v)}{\JNA_{g,\bT}(\cF_v)} = \lim_{i\to\infty} \frac{\DNA_g(\cF_{v_i})}{\JNA_{g,\bT}(\cF_{v_i})}= 0,
\]
which implies $\DNA_g(\cF_v)=A_{X,\Delta}(v)-S_g(v_0;v)=0$. Since $\delta_g(X,\Delta,\xi)\ge 1$ by Corollary \ref{c:ss=delta>=1}, we see that $\delta_g(X,\Delta,\xi) = 1$ and $v$ computes $\delta_g (X,\Delta,\xi)$.
\end{proof}

We have used the following statement in the above proof.

\begin{lem} \label{l:S_g continuous}
Let $Y\to X$ be a proper birational map with $Y$ regular and $E:= \sum_{i=1}^d E_i$ a reduced simple normal crossing divisor on $Y$. 
Then the function $S_{g}(v_0;\, \cdot \, )$ is continuous on $\QM(Y,E)$. 
\end{lem}

We will deduce the result from  the continuity of $S(\cdot  )$ on $\QM(Y,E)$ shown in \cite{BLX19}.

\begin{proof}
Fix a point $\eta \in \Supp(E)$ and local coordinates $y_1,\ldots, y_r \in \cO_{Y,\eta}$ so that each $y_j$ cuts out an irreducible component of $E$ at $\eta$. For $\alpha=( \alpha_1,\ldots, \alpha_r)\in \R_{\geq0}^{r}$, write $v_\alpha\in \Val_X$ for the associated quasi-monomial valuation satisfying $v_\alpha(y_j) = \alpha_j$. 

To prove the continuity of $S_g(v_0;\, \cdot\, )$ on $\R_{\geq0}^r$, fix  a convergent sequence $\alpha^{i}$ in $\R_{\geq 0}^{r}$ and set $ \alpha: = \lim_i \alpha^{i}$. We aim to show $\lim_i S_g(v_0;v_i) = S_g(v_0;v)$, where $v_i: = v_{\alpha^i}$ and $v : = v_{\alpha}$. 

First, note that $S_g(v_0;c  w)= cS_g(v_0;w) $ for all $w\in \Val_X^\circ$ and $c\in \R_{>0}$. 
Therefore, after rescaling the $\alpha^i$ and removing finitely many terms, we may assume the sequence $(\alpha^i)$ is non-increasing. Hence, $v_i \geq v_{i+1} \geq v$ for all $i$.  
Next, consider the Okounkov body $\Sigma \subset \R^n$ induced by the good valuation form Lemma \ref{l:goodval}. Write $G_i$ and $G$ for the concave functions $\Sigma \to \R_{\geq 0}$ induced by the filtrations $\cF_{v}$ and $\cF_{v_i}$ (see \cite[Section 2.5]{BJ20}) and set $\vol_g(\Sigma)= \int_{\Sigma} g\, \rd \rho $. 
Note that
\begin{equation}\label{e:Sgvi}
|S_g(v_0;v_i) - S_g(v_0;v) | 
= \left| \frac{1}{ \vol_g (\Sigma) } \int_{\Sigma }g G_i \, \rd \rho - 
 \frac{1}{ \vol_g (\Sigma) } \int_{\Sigma } g G \, \rd \rho
 \right | 
 = \frac{1}{\vol_g(\Sigma)} \int_{\Sigma} g |G_i-G| \, \rd \rho  
,\end{equation}
where the second equality uses that $G_i\leq G$.
Using that    $\lim_i S_1(v_0; v_i) = S_{1}(v_0;v)$ by \cite[Proposition 2.4]{BLX19} and \eqref{e:Sgvi}, we see
$(G_i)$ converges to $G$ a.e.
Therefore, the dominated convergence theorem implies
\[
\lim_{i\to \infty} |S_g(v_0;v_i) - S_g(v_0;v) | = 
\lim_{i \to \infty }
  \frac{1}{\vol_g(\Sigma)} \int_{\Sigma} g |G_i-G| \, \rd \rho  
 =0
,\]
which completes the proof.
\end{proof}

\section{Finite generation}

\subsection{A valuative criterion for $\tbeta_{X,\Delta}$-minimizers}\label{ss-HNAval}

%  We define the relative $\tS$-invariant as follows.

% \begin{defn}
% Let $v_0, v\in \Val_X^\circ$. Let $(s_1, \cdots, s_{N_m})$ be a basis of $R_m$ that is compatible with both $v_0$ and $v$. We define 
% \begin{align*}
% \bfQ_m(v_0) :=\frac{1}{N_m}\sum_{i=1}^{N_m} e^{-\frac{v_0(s_i)}{mr}}\quad \textrm{and} \quad\tS_m(v_0;v) := \frac{1}{N_m \bfQ_m(v_0)} \sum_{i=1}^{N_m} \frac{v(s_i)}{mr}\cdot e^{-\frac{v_0(s_i)}{mr}}.
% \end{align*}
% Let $\tS(v_0;v):=\lim_{m\to \infty} \tS_m(v_0;v)$.
% \end{defn}

In this section we give a valuative criterion for valuations computing $h$-invariant inspired by \cite{XZ20nvol} in terms of weighted stability thresholds. Let us recall that in Corollary \ref{c-quasiunique}, we know the valuation computing $h(X,\Delta)$ is quasi-monomial. Thus we can apply the construction in Section \ref{ss-weighteddelta}. 
We use the following notation: for any quasi-monomial $v_0$ and $v\in \Val_X^\circ$, let $\tS(v_0;v):=S_g(v_0;v)$ where $g(x)=e^{-x}$.  By \cite{HL20-uniqueness} and Theorem \ref{t:Hunique}, the unique valuation computing $h(X,\Delta)$ is trivial if and only if $(X,\Delta)$ is K-semistable. Hence, throughout this subsection, we assume that $(X,\Delta)$ is K-unstable, i.e. $h(X,\Delta)<0$.

\begin{thm} \label{t:Hmin induce K-ss triple}
A quasi-monomial valuation $v_0\in \Val_{X}^\circ$ computes $h(X,\Delta)$ if and only if $A_{X,\Delta}(v)\geq \tS(v_0;v)$ for any valuation $v\in \Val_X^\circ$ and $A_{X,\Delta}(v_0)=\tS(v_0;v_0)$.
\end{thm}

\begin{proof}
We first show the ``if'' part. It suffices to show that $\tbeta_{X,\Delta}(v) \geq \tbeta_{X,\Delta}(v_0)$ for any valuation $v\in \Val_X^\circ$. Let $\cF_0:=\cF_{v_0}$ and $\cF_1:=\cF_{v}$. Let $\bnu$ denote the compatible measure of $\cF_0$ and $\cF_1$ from Section \ref{s:measureR2}. From the definitions, we know that 
\begin{equation}\label{eq:tS-int}
% \bfQ(v_0)=\int_{\bR^2} e^{-x} d\bnu, \quad 
\tS(v_0;v_0)=\frac{\int_{\bR^2} x e^{-x} d\bnu}{\int_{\bR^2} e^{-x} d\bnu}, \quad \textrm{and}\quad \tS(v_0;v)=\frac{\int_{\bR^2} y e^{-x} d\bnu}{\int_{\bR^2} e^{-x} d\bnu}.
\end{equation} Consider the following function $f:[0,1]\to \bR$ given by 
\[
f(t):= (1-t)A_{X,\Delta}(v_0) +tA_{X,\Delta}(v) +\log \int_{\bR^2} e^{-(1-t)x-ty} d\bnu.
\]
It is clear that $f(0)=\tbeta_{X,\Delta}(v_0)$ and $f(1)=\tbeta_{X,\Delta}(v)$.
By H\"older's inequality as in the proof of Proposition \ref{p:Sconvex}, we know that $f(t)$ is convex in $t$. Moreover, we have
\begin{align*}
f'(0) & = A_{X,\Delta}(v) - A_{X,\Delta}(v_0) + \frac{\int_{\bR^2} (x-y)e^{-x} d\bnu}{\int_{\bR^2} e^{-x} d\bnu}\\
& = (A_{X,\Delta}(v)- \tS(v_0;v)) - (A_{X,\Delta}(v_0)-\tS(v_0;v_0)) \geq 0.
\end{align*}
Thus $f(1)\geq f(0)$ and the ``if'' part is proved.

Next, we show the ``only if'' part.  Let $(\cF_t)_{t\in [0,1]}$ be the geodesic of filtrations connecting $\cF_0$ and $\cF_1$. 
Since $v_0$ computes $h(X,\Delta)$, we know that $\HNA(\cF_t)\geq \HNA(\cF_0)=f(0)$ for any $t\in [0,1]$. Recall that
\[
\HNA(\cF_t)= \LNA(\cF_t)-\tS(\cF_t)=\LNA(\cF_t)+ \log \int_{\bR^2} e^{-(1-t)x-ty} d\bnu.
\]
By Proposition \ref{p:Lconvex} and 
Corollary \ref{c:LNA<=A}, we obtain
\[
\LNA(\cF_t) \leq (1-t)\LNA(\cF_0)+t\LNA(\cF_1)\leq (1-t) A_{X,\Delta}(v_0)+t A_{X,\Delta}(v). 
\]
Hence we have $\HNA(\cF_t)\leq f(t)$ which implies that $f(t)\ge f(0)$ for all $t\in [0,1]$ and thus $f'(0)\geq 0$, i.e.
\[
A_{X,\Delta}(v)- \tS(v_0;v) \geq  A_{X,\Delta}(v_0)-\tS(v_0;v_0).
\]
Since $v\in \Val_{X}^\circ$ is arbitrary, the above inequality remains true if we replace $v$ by $\lambda v$ for any $\lambda\in \bR_{>0}$, i.e. 
\[
\lambda (A_{X,\Delta}(v) - \tS(v_0;v))\geq A_{X,\Delta}(v_0)-\tS(v_0;v_0).
\]
Thus we have $A_{X,\Delta}(v) \geq  \tS(v_0;v)$ for any $v\in \Val_X^\circ$ and $A_{X,\Delta}(v_0)\leq \tS(v_0;v_0)$, which implies that $A_{X,\Delta}(v_0)= \tS(v_0;v_0)$.
This finishes the proof.
\end{proof}

The previous theorem immediately implies the following corollary. 

\begin{cor} \label{c:hmin compute delta}
Let $g(x)=e^{-x}$ and let $v_0\in\Valc$ be the valuation computing $h(X,\Delta)$. Then $\delta_g(X,\Delta,v_0)=1$ and is computed by $v_0$.
\end{cor}

% \begin{rem}
% The above corollary can be rephrased as saying that $v_0$ computes $h(X,\Delta)$ if and only if $(X,\Delta,v_0)$ is $g$-Ding semistable where $g=e^{-x}$. On the other hand, by Theorem \ref{t:Hunique}, the unique valuation $v_0\in\Valc$ that computes $h(X,\Delta)$ is necessarily $\bT$-invariant for any torus $\bT<\Aut(X,\Delta)$. Combined with the proof of \cite{HL20-uniqueness}*{Theorem 4.9} (or the proof of Theorem \ref{t:f.g. for delta_g minimizer} below), we also know that $v_0$ is an lc place of some complement. Thus from the above proof of Theorem \ref{t:Hmin induce K-ss triple}, we see that $v_0$ minimizes $\tbeta_{X,\Delta}(v)$ among $\bT$-invariant valuations as long as $A_{X,\Delta}(v)\ge \tS(v_0;v)$ for all $v\in\ValT$ that are lc places of complements. By \cite{BLX19}*{Theorem A.2}, the latter condition is equivalent to $\DNA_g(\cX,\Delta_{\cX},\cL)\ge 0$ for all $\bT$-equivariant normal test configurations of $(X,\Delta)$. It follows that to test the K-semistability of a triple $(X,\Delta,\xi)$ (where $\xi\in N_\bR$), it is enough to use $\bT$-equivariant test configurations. Therefore, our definition of K-semistability for triples agrees with those in \cites{BWN14,HL20-uniqueness}.
% \end{rem}

\subsection{$\delta_g$-minimizers}

In this section, we fix a continuous function $g\colon \bR\to \bR_{>0}$, a torus $\bT<\Aut(X,\Delta)$ and a quasi-monomial valuation $v_0\in\ValT$. Let $N=\Hom(\bG_m,\bT)$ and $N_\bR=N\otimes_\bZ \bR$ as before.

\begin{que} 
Assume that $\delta_{g,\bT}(X,\Delta,v_0)\le 1$. Let $v\in\ValT$ be a valuation that computes $\delta_{g,\bT}(X,\Delta,v_0)$. Is the associated graded ring $\gr_v R:=\gr_{\cF_v} R$ finitely generated?
\end{que}

We give an affirmative answer in two special cases, which is enough for our applications.

\begin{thm} \label{t:f.g. for delta_g minimizer} 
Let $v\in\ValT$ be a quasi-monomial valuation that computes $\delta_{g,\bT}(X,\Delta,v_0)$. Assume that $\delta_{g,\bT}(X,\Delta,v_0)\le 1$, and that $v=(v_0)_\xi$ or $v_0=\wt_\xi$ for some $\xi\in N_\bR$. Then $\gr_v R$ is finitely generated.
\end{thm}

For the proof we first recall a statement that can be extracted from the proof of \cite[Lemma 3.1]{LXZ-HRFG}.

\begin{lem}\label{l:complement}
Let $v$ be a quasi-monomial valuation on $X$. Assume that there exists a sequence of $\bQ$-divisors $D_m$ $(m\in\bN)$ such that $(X,\Delta+D_m)$ is lc and $-(K_X+\Delta+D_m)$ is semiample for all $m$, and that $\lim_{m\to\infty} v(D_m)=A_{X,\Delta}(v)$. Then $v$ is an lc place of $(X,\Delta+\Gamma)$ for some $\bQ$-complement $\Gamma$ of $(X,\Delta)$.
\end{lem}

\begin{proof}
We only sketch the proof since the argument is almost the same as in \cite[Lemma 3.1]{LXZ-HRFG}. After rescaling, we assume that $A_{X,\Delta}(v)=1$. Since $v$ is quasi-monomial, we have $v\in\QM(Y,E)$ for some log smooth model $(Y,E)\to (X,\Delta)$. Let $\fa_m = \fa_m(v)$ ($m\in\bN$) be the valuation ideals. By the proof of \cite[(3.1)]{LXZ-HRFG} (which only uses the fact that $v$ is quasi-monomial), we know that for any $\varepsilon\in(0,1)$, there exists $\varepsilon_0>0$ and divisorial valuations $v_i=\ord_{F_i}\in\QM(Y,E)$ ($i=1,\cdots,r$) such that $v$ is in the convex hull of $v_1,\cdots,v_r$ and $A_{X,\Delta+\fa_\bullet^{1-\varepsilon_0}}(F_i)<\varepsilon$ for all $i$. By assumption, we have $v(D_m)>1-\varepsilon_0$ for sufficiently large $m$ and for such $m$ we obtain
\begin{equation} \label{e:A small}
    A_{X,\Delta+D_m}(F_i)\le A_{X,\Delta+\fa_\bullet^{1-\varepsilon_0}}(F_i)<\varepsilon.
\end{equation}
By \cite{BCHM}*{Corollary 1.4.3} we get a $\bQ$-factorial birational model $p\colon \tX\to X$ that extracts exactly the divisors $F_i$. By assumption, all $(X,\Delta+D_m)$ ($m\in\bN$) have $\bQ$-complements. Together with \eqref{e:A small} this implies that $(\tX,p_*^{-1}\Delta+(1-\varepsilon)\sum_{i=1}^r F_i)$ has $\bQ$-complements as well. Using \cite[Lemma 3.2]{LXZ-HRFG}, we conclude that $\bQ$-complements also exist for $(\tX,p_*^{-1}\Delta+\sum_{i=1}^r F_i)$ as long as $\varepsilon$ is sufficiently small. Since $v$ is in the complex hull of $\ord_{F_i}$, this yields a $\bQ$-complement $\Gamma$ of $(X,\Delta)$ that has $v$ as an lc place.
\end{proof}

\begin{lem} \label{l:D>=cG}
There exists some constant $c>0$ such that $S_g(v_0;G)>c$ for all effective $\bQ$-divisors $G\sim_\bQ -(K_X+\Delta)$ on $X$. In particular, for any $m\gg 0$ and any $g$-weighted $(m,v_0)$-basis type divisor $D$ that is compatible with $G$, we have $D\ge cG$.
\end{lem}

\begin{proof}
Let $T=T_{X,\Delta}(v_0)<\infty$, and let
\[
c_0=\frac{\inf_{x\in[0,T]} g(x)}{\sup_{x\in[0,T]} g(x)}>0.
\]
Let $\bnu$ denote the compatible DH measure associated to $\cF_{v_0}$ and $\cF_G$ as in Section \ref{s:measureR2}. Then $\bnu$ is supported in $[0,T]\times \bR$ and we have
\[
S_g(v_0;G)=\frac{\int_{\bR^2} yg(x)\rd \bnu}{\int_{\bR^2} g(x)\rd \bnu}\ge c_0\cdot  \frac{\int_{\bR^2} y\rd \bnu}{\int_{\bR^2} \rd \bnu} = c_0\cdot S_{X,\Delta}(G)=\frac{c_0}{n+1}
\]
where the last equality is by \cite[Lemma 2.20]{LXZ-HRFG}. Thus we may take e.g. $c=\frac{c_0}{3n}$.
\end{proof}

We are now ready to prove Theorem \ref{t:f.g. for delta_g minimizer}.

\begin{proof}[Proof of Theorem \ref{t:f.g. for delta_g minimizer}]
The plan is to use Lemma \ref{l:complement} to show that $v$ is a monomial lc place of a special complement (in the sense of \cite[Definition 3.3]{LXZ-HRFG}), and then apply \cite[Theorem 4.2]{LXZ-HRFG} to get the finite generation. To this end, let $\pi:(Y,E=\sum_{i=1}^r E_i)\to (X,\Delta)$ be a $\bT$-equivariant log smooth model such that $\QM(Y,E)$ is a simplicial cone whose interior contains $v$, $C_Y(v)=\cap_{i=1}^r E_i$, and there is a $\pi$-exceptional and $\pi$-ample $\bQ$-divisor $-F$ on $Y$. 

Let $G_Y$ be a $\bT$-invariant $\bQ$-divisor in the ample $\bQ$-linear system $|-\pi^*(K_X+\Delta)-\varepsilon F|_{\bQ}$ ($0<\varepsilon\ll 1$) whose support does not contain $C_Y(v_0)$ (such $G_Y$ exists, because there is some $\bT$-invariant element $\Gamma$ in $H^0(Y,m(-\pi^*(K_X+\Delta)-\varepsilon F))$ for sufficiently divisible $m$ with $v_0(\Gamma)\neq 0$). Let $G:=\pi_*G_Y$. For any $m\in\bN$, let $D_m\in |-K_X-\Delta|_\bR$ be a $\bT$-invariant $g$-weighted $(m,v_0)$-basis type divisor that is also compatible with both $v$ and $G$. Such divisors exist because:
\begin{itemize}
    \item both $v$ and $G$ are $\bT$-invariant (so we can choose compatible basis in each individual piece in the weight decomposition), and
    \item by our assumption, any $\bT$-invariant basis that is compatible with both $v$ and $G$ is automatically compatible with $v_0$.
\end{itemize}
By Lemma \ref{l:D>=cG}, there exists some $c\in\bQ_+$ such that $D_m\ge cG$ for all $m\gg 0$. Let $\delta_m=\min\{\delta_{g,\bT,m}(X,\Delta,v_0),1\}$. Then $(X,\Delta+\delta_m D_m)$ is lc and $\lim_{m\to\infty} \delta_m = \delta_{g,\bT}(X,\Delta,v_0)$ by Lemma \ref{l:delta=lim delta_m} and the assumption that $\delta_{g,\bT}(X,\Delta,v_0)\le 1$. 

Since $v$ computes $\delta_{g,\bT}(X,\Delta,v_0)$, we also have
\[
\lim_{m\to\infty} \delta_m v(D_m) = \delta_{g,\bT}(X,\Delta,v_0) S_g(v_0;v) = A_{X,\Delta}(v).
\]
Note that by construction $D_m=\sum \lambda_i D_m^{(i)}$ for some $\lambda\in\bR_+$ and some effective $\bQ$-divisors $D_m^{(i)}\sim_\bQ -(K_X+\Delta)$. Thus by perturbing the coefficients $\lambda_i$, for each $m\gg 0$ we get a $\bQ$-divisor $D'_m=\sum \lambda'_i D_m^{(i)}$ such that $\delta_m D_m > D'_m\ge \frac{1}{2}cG$ and $\lim_{m\to\infty} v(D'_m)=A_{X,\Delta}(v)$. It follows that $(X,\Delta+cG+\tD_m)$ is lc and $-(K_X+\Delta+cG+\tD_m)$ is ample where $\tD_m=D'_m-\frac{1}{2}cG$. % Since $v$ is quasi-monomial by Lemma \ref{l:delta_g min qm}, we may then apply 
By Lemma \ref{l:complement}, we see that $v$ is an lc place of $(X,\Delta+\frac{1}{2}cG+\Gamma)$ for some $\bQ$-complement $\Gamma$ of $(X,\Delta+\frac{1}{2}cG)$. Recall the $\pi_*^{-1}G$ is ample and does not contain $C_Y(v)$. By \cite[Theorem 4.2]{LXZ-HRFG}, this implies that $\gr_v R$ is finitely generated.
\end{proof}

\begin{cor}\label{c-hfg}
Assume that $(X,\Delta)$ is not K-semistable. Let $v_0\in\Valc$ be the unique valuation computing $h(X,\Delta)$. Then $\gr_{v_0} R$ is finitely generated.
\end{cor}

\begin{proof}
By \cite[Theorem 1.5]{HL20-uniqueness}, $v_0$ is quasi-monomial. So the result follows immediately from Corollary \ref{c:hmin compute delta} and Theorem \ref{t:f.g. for delta_g minimizer} (with $g(x)=e^{-x}$ and $\bT=\{1\}$).
\end{proof}

\begin{proof}[Proof of Theorem \ref{t-finitegen}]
By Corollary \ref{c-hfg}, we know that $v$ yields a special $\bR$-test configuration in the sense of \cite{HL20-uniqueness}*{Definition 2.8}. Thus by \cite{HL20-uniqueness}*{Theorem 1.6}, we know that $(X_0,\Delta_0,\xi_v)$ is a K-semistable triple. 
\end{proof}

\begin{cor}\label{c-deltafg}
Any quasi-monomial valuation $v\in \Val_{X}^{\bT,\circ}$ computing $\delta_g(X,\Delta,\xi)$, where $\xi\neq 0$, has a finitely generated associated graded ring. 
\end{cor}

\begin{proof}
In view of Theorem \ref{t:f.g. for delta_g minimizer}, it suffices to show that $\delta_g(X,\Delta,\xi)\le 1$. Indeed, we will prove a stronger statement: 
\begin{equation} \label{e:delta_T<=1}
    \delta_{g,\bT,m} (X,\Delta,v_0)\le 1
\end{equation}
for all $v_0\in\ValT$, $m\in\bN$ and all positive functions $g\in C^0(\bR)$, as long as $\bT\neq\{1\}$. To see this, let $D$ be a $\bT$-invariant $g$-weighted $(m,v_0)$-basis type divisor. If $(X,\Delta+D)$ is klt, then after perturbing the coefficients of $D$ as in the proof of Theorem \ref{t:f.g. for delta_g minimizer}, we get a $\bT$-invariant $\bQ$-divisor $\tD>D$ proportional to $-(K_X+\Delta)$ such that $(X,\Delta+\tD)$ is still klt and $K_X+\Delta+\tD$ is ample. So such pairs have finite automorphism groups %(see e.g. \cite[Proposition 6.5]{KP17}) 
and this is a contradiction as $\bT<\Aut(X,\Delta+\tD)$ by construction. Thus $(X,\Delta+D)$ is not klt and $\lct(X,\Delta;D)\le 1$. This proves \eqref{e:delta_T<=1}.
\end{proof}

\begin{cor}\label{c-YTD}
% A triple $(X,\Delta,\xi)$ is $g$-Ding polystable if and only if it is reduced uniformly Ding stable. In particular, when $\bk=\bC$, $(X,\Delta,\xi)$ admits a K\"ahler-Ricci $g$-soliton if and only if it is $g$-Ding polystable.
Any $g$-Ding polystable triple $(X,\Delta,\xi)$ is also reduced uniformly $g$-Ding stable. In particular, it admits a K\"ahler-Ricci $g$-soliton when $\bk=\bC$.
\end{cor}

\begin{proof}
The proof is very similar to that of \cite{LXZ-HRFG}*{Theorem 5.2}. Let $\bT<\Aut(X,\Delta,\xi)$ be a maximal torus such that $\xi\in N_\bR$. % By the argument in \cite{HL20-uniqueness}*{Section 8}, we know that $g$-Ding polystability is equivalent to $\bT$-equivariant $g$-Ding polystability (it is shown there that it suffices to check $\bT$-equivariant test configurations), thus reduced uniform $g$-Ding stability implies $g$-Ding polystability. It remains to prove the reverse direction. 
Assume to the contrary that $(X,\Delta,\xi)$ is $g$-Ding polystable but not reduced uniformly $g$-Ding stable. Then by Proposition \ref{p:red uks minimizer}, we know that $\delta_g(X,\Delta,\xi)=1$ is computed by some quasi-monomial valuation $v\in\ValT$ that is not of the form $\wt_\eta$. By Corollary \ref{c-deltafg}, the associated graded ring $\gr_v R$ is finitely generated. Let $\pi\colon (Y,E)\to (X,\Delta)$ be a $\bT$-equivariant log smooth model such that $\QM(Y,E)$ is a simplicial cone containing $v$ and that its dimension is the same as the rational rank of $v$ . As in \cite{Xu20-HRFG}*{Claim 3.10}, this implies that in a neighbourhood of $v$ in $\QM(Y,E)$, the function $w\mapsto S_g(v_0;w)$ is linear, and we have $\gr_wR\cong \gr_vR$. 

Thus $\delta_g(X,\Delta,\xi)$ is also computed by some $\bT$-invariant divisorial valuation $w\in\QM(Y,E)$ that is sufficiently close to $v$. % (in the minimal affine space of $\QM(Y,E)$ containing $v$)
In particular, $\DNA_g(\cF_w)\le A_{X,\Delta}(w)-S_g(v_0;w)=0$ (the first inequality is by Corollary \ref{c:LNA<=A}) and $w\neq\wt_\eta$ for any $\eta\in N_\bR$. It induces a non-product type $\bT$-equivariant test configuration $(\cX,\Delta_{\cX},\cL)$ such that $\DNA_g(\cX,\Delta_{\cX},\cL)\le 0$. This contradicts the $g$-Ding polystability of $(X,\Delta,\xi)$ and proves the first part of the corollary. The remaining part follows from Theorem \ref{t-HL-YTD-soliton}.
\end{proof}

% Since $(X,\Delta,\xi)$ is reduced uniformly $g$-Ding stable, Theorem \ref{t-HL-YTD-soliton} implies that $(X,\Delta,\xi)$ admits a K\"ahler-Ricci $g$-soliton. 

\begin{proof}[Proof of Theorem \ref{t-YTD}]
By \cite[Theorem 1.5]{BWN14}, $(X,\Delta,\xi)$ is K-polystable if it admits a K\"ahler-Ricci soliton. Thus the result follows immediately from Theorem \ref{t-HL-YTD-soliton} and Corollary \ref{c-YTD} by setting $g(x)=e^{-x}$.
\end{proof}
\begin{rem}\label{rem-equiv}
The above proof that reduced uniform Ding stability implies K-polystability uses  K\"ahler-Ricci solitons, but it can be proved algebraically. Though, there is some subtlety, since the data of a triple $(X,\Delta,\xi)$ includes a torus $\T$ so that $\xi \in N_{R} : = \Hom(\G_m,\T)\otimes_{\Z}\R$ and $\T$ is not necessarily maximal. While the K-polystability  of $(X,\Delta,\xi)$ is with respect to $\T$,  reduced uniform Ding stability is defined using a maximal torus $\T < \T^{\max} < \Aut(X,\Delta)$. 

To prove the equivalence, observe that if $(X,\Delta,\xi)$ is reduced uniformly Ding stable, then
it is K-polystable with respect to $\T^{\max}$ by \cite{HL20-YTD}*{Proposition 5.16}. 
 To show it is K-polystable with respect to $\T$, first by \cite{HL20-uniqueness}*{(168) or (189)},
  it follows that $(X,\Delta,\xi)$ is K-semistable with respect to $\T$. 
 Then by verbatim applying the proof of \cite{LWX-tangentcone}*{Theorem 3.7} (see also \cite{HL20-uniqueness}*{Section 8}), we know the  K-polystablity of $(X,\Delta,\xi)$ with respect to $\T^{\max}$ implies the K-polystability of $(X,\Delta,\xi)$ with respect to $\T$.
\end{rem}
 
\begin{proof}[Proof of Theorem \ref{t-optimaldeg}]
It is a combination of Theorem \ref{t-finitegen}, \cite{HL20-uniqueness}*{Theorem 1.3}, and Theorem \ref{t-YTD}.
\end{proof}

As an application, we show the following theorem, which generalizes \cite{WZ04} from the smooth case to general toric log Fano pairs. See also \cite{HL20-YTD}*{Section 8}. 

\begin{thm}
For any toric log Fano pair $(X,\Delta)$ over $\bC$, there exists a vector $\xi\in N_{\mathbb R}:={\rm Hom}(\mathbb{G}_m, \bT)\otimes_{\mathbb{Z}}\mathbb{R}$ where $\bT$ is the maximal torus acting on $(X,\Delta)$, such that $(X,\Delta,\xi)$ admits a K\"ahler-Ricci soliton. 
\end{thm}
\begin{proof}
By Corollary \ref{c-quasiunique}, there is a unique valuation $v\in \Val_X$ which computes $h(X,\Delta)$. 
By the uniqueness, $v$ is $\bT$-invariant i.e. $v=\wt_{\xi}$ for some $\xi\in N_\R$. 
Therefore, the K-semistable triple produced in Theorem \ref{t-finitegen} is $(X,\Delta, \xi)$. 

Since $(X,\Delta,\xi)$ is K-semistable and toric, it is reduced uniformly Ding stable by Theorem \ref{l:wstc}. Indeed, condition (3) of the theorem holds trivially, since any $w\in \Val_{X}^{\bT,\circ}$ is of the form $w=\wt_{\eta}$ for some $\eta\in N_{\R}$ and, hence, satisfies $\JNA_{g,\T}(w) = 0$ where $g=e^{-x}$. Therefore, $(X,\Delta,\xi)$ admits a K\"ahler-Ricci soliton by \cite{HL20-YTD}.
%By \cite{HL20-uniqueness}*{Section 8}, the degeneration from $(X,\Delta,\xi)$ to the unique K-polystable triple is $\bT$-equivariant, which implies that $(X,\Delta,\xi)$ itself is K-polystable, as it does not have a non-trivial $\bT$-equivariant degeneration. Thus it admits a K\"ahler-Ricci soliton by Theorem \ref{t-YTD}.
%\footnote{ \HB{It seems more direct to me to avoid using \cite{HL20-uniqueness}*{Section 8}. Also, that reference does not mention the $\T$-equivariant result (of course, it can be deduced using the argument in LWX).}}
%\footnote{\HB{We could replace the second paragraph with:``Since $(X,\Delta;\xi)$ is K-semistable and toric, it is reduced uniformly Ding stable by Theorem \ref{l:wstc}. Indeed, condition (3) of the theorem holds trivially, since any $w\in \Val_{X}^{\bT,\circ}$ is of the form $w=\wt_{\eta}$ for some $\eta\in N_{\R}$ and, hence, satisfies $\JNA_{g,\T}(w) = 0$ where $g=e^{-x}$. Therefore, $(X,\Delta;\xi)$ admits a K\"ahler-Ricci soliton by \cite{HL20-YTD}. ''}}
\end{proof}

\section{Moduli stack}

In this section, we will prove Theorem \ref{t-stack}. It suffices to verify the boundedness and openness, see Theorem \ref{t-bounded} and Theorem \ref{t-openness}. 

\begin{thm}\label{t-moduli}
For a fixed dimension $n$, volume $V$, a positive integer $N$ and a negative constant $h_0<0$, families of $n$-dimensional log Fano pairs $(X,\Delta)$ with $(-K_X-\Delta)^n=V$, $N\Delta$ integral and $h(X,\Delta)\ge h_0$ are parametrized by an Artin stack 
$\mathcal{M}^{\rm Fano}_{n,V,N,h^+_0}$ of finite type.  
\end{thm}

The following result gives the boundedness.

\begin{prop}\label{p-bounded}
Let $(X,\Delta)$ be a log Fano pair and let $c\in \bR$. Assume that $\tbeta_{X,\Delta}(v)\ge c$ for all divisorial valuations $v$ on $X$. Then $\alpha(X,\Delta)\ge \alpha$ for some constant $\alpha>0$ that only depends on $c$ and $\dim (X)$.
\end{prop}

Here, $\alpha(X,\Delta)$ denotes Tian's $\alpha$-invariant. In the proof, we use that $\alpha(X,\Delta)$ equals  $\inf_{v}\frac{A_{X,\Delta}(v)}{T(v)}$, where the infimum runs through all divisorial valuations on $X$; see  e.g. \cite{BJ20}.

\begin{proof}
It suffices to find some constant $M=M(c,n)>0$ that only depends on $c$ and $n=\dim(X)$ such that $T(v)\le M$ for all divisorial valuations $v$ with $A_{X,\Delta}(v)=1$.
For now fix any such $v$ and let $\cF=\cF_v$, $T=T_{X,\Delta}(v)$, $\tS=\tS(v)$. We will apply a modified argument from \cite{BJ20}*{Lemma 2.6}. Let $f(\lambda)=\frac{ \vol( V_\bullet^{ (\la)}) }{L^n}$. Note that $f(0)=1$, $f(T)=0$ and $\Dh_{\cF} = -f'(\lambda) \rd \lambda$. By \cite{Laz04}*{Theorem 11.4.9}, the function $\lambda\mapsto f(\lambda)^{\frac{1}{n}}$ is concave on $(0,T)$, thus
\[
f(\lambda)\ge \left(1-\frac{\lambda}{T}\right)^n.
\]
On the other hand, integration by parts yields
\[
e^{-\tS} = \int_0^T e^{-\lambda} \Dh_{\cF}(\rd \lambda) = 1 - \int_0^T e^{-\lambda} f(\lambda) \rd \lambda,
\]
hence $e^{-\tS}\le 1 - \int_0^T e^{-\lambda} \left(1-\frac{\lambda}{T}\right)^n \rd \lambda$. We may further rewrite the right hand side as %\footnote{\HB{Should the left hand side below instead be $ 1 - \int_0^T e^{-\lambda} \left(1-\frac{\lambda}{T}\right)^n \rd \lambda$? (otherwise, we don't get equality on the first line below.}}
\begin{align*}
    \int_0^\infty e^{-\lambda} \rd \lambda - \int_0^T e^{-\lambda} \left(1-\frac{\lambda}{T}\right)^n \rd \lambda & = \int_0^T e^{-\lambda} \left[ 1-\left(1-\frac{\lambda}{T}\right)^n \right] \rd \lambda + \int_T^\infty e^{-\lambda} \rd \lambda \\
    & \le \int_0^T \frac{n\lambda}{T} e^{-\lambda} \rd \lambda + \int_T^\infty e^{-\lambda} \rd \lambda \\
    & \le \frac{n}{T}+e^{-T}.
\end{align*}
By assumption, $1-\tS = \tbeta(v)\ge c$, hence $e^{-\tS}\ge e^{c-1}$. It follows that
\[
\frac{n}{T}+e^{-T} \ge e^{c-1}.
\]
From this we deduce that $T$ is bounded from above by some constant that only depends on $c$ and $n$. The proof is now complete.
%%%%%%%%%%% original wrong proof
% \cite{Fuj-alpha}*{Proposition 3.2}. Clearly,
% \[
% \int_0^T e^{-\lambda} \Dh_{\cF}(\rd \lambda) = e^{-b}
% \]
% for some $b\in (0,T)$. By \cite{LM09}*{Corollary C} and \cite{ELMNP}*{Theorem A}, we know that $\Dh_{\cF}(\rd \lambda)=f(\lambda)^{n-1}\rd \lambda$ for some concave function $f>0$ on $(0,T)$.
% In particular, $f(\lambda)\ge \frac{T-\lambda}{T-b}f(b)$ when $\lambda\in[b,T)$ (this right side of the inequality represents the line joining $(b,f(b))$ and $(T,0)$) and $f(\lambda)\le \frac{T-\lambda}{T-b}f(b)$ when $\lambda\in(0,b)$. It follows that
% \[
% \int_0^T (e^{-\lambda}-e^{-b})f(\lambda)^{n-1}\rd \lambda \le \int_0^T (e^{-\lambda}-e^{-b})\left(\frac{T-\lambda}{T-b}f(b)\right)^{n-1} \rd \lambda.
% \]
% On the other hand,
% \[
% \int_0^T (e^{-\lambda}-e^{-b})f(\lambda)^{n-1}\rd \lambda = \int_0^T (e^{-\lambda}-e^{-b}) \Dh_{\cF}(\rd \lambda) = 0
% \]
% by our choice of $b$. Thus we obtain $\int_0^T (e^{-\lambda}-e^{-b})(T-\lambda)^{n-1} \rd \lambda \ge 0$. By assumption, we also have
% \[
% e^{-b}=\int_0^T e^{-\lambda} \Dh_{\cF}(\rd \lambda) = e^{\tbeta_{X,\Delta}(v)-A_{X,\Delta}(v)}\ge e^{c-1}.
% \]
% Therefore,
% \[
% e^{c-1}\le e^{-b}\le \frac{\int_0^T e^{-\lambda} (T-\lambda)^{n-1} \rd \lambda}{\int_0^T (T-\lambda)^{n-1} \rd \lambda} = O\left(\frac{1}{T}\right)
% \]
% by a direct integration-by-parts calculation. 
\end{proof}

\begin{thm}\label{t-bounded}
Fixed positive integers $n$, $N$, a positive number $V_0$ and a constant $h_0 $. Denote by $\mathcal{P}$ the set of $n$-dimensional log Fano pairs $(X,\Delta)$ with $N\cdot \Delta$ integral which satisfy $(-K_X-\Delta)^n\ge V_0$ and $h(X,\Delta)\ge h_0$. Then $\mathcal{P}$ is bounded.
\end{thm}

\begin{proof}
For fixed $\alpha_0>0$, the set of log Fano pairs $(X,\Delta)$ with $n=\dim(X)$, $N\cdot \Delta$ integral and $(-K_X-\Delta)^n \geq V_0$, and $\alpha(X,\Delta)>\alpha_0$ are bounded by  \cite{Che18} \cite{Jia20} \cite{XZ20nvol}. Applying Proposition \ref{p-bounded} then completes the proof.
\end{proof}

Next we will prove the openness. It suffices to show the following theorem.
\begin{thm}\label{t-openness}
Let $(X,\Delta)\to B$ be a locally stable family of log Fano pairs over a scheme $B$ of finite type. 
Then 
$$h\colon t\to h(X_t,\Delta_t), \ \ \ t\in B$$ 
is a constructible and lower semicontinuous. 
%Assume there is a torus group $T$ which fiberwisely acts on $(X,\Delta)$. Let $\xi\in {\rm Hom}(\mathbb{G}_m, T)_{\mathbb R}$ be a (real) coweight. Then $ $
\end{thm}
\begin{proof}By passing to a resolution of $B_{\rm red}$, we may assume $B$ is smooth. By the proof of Theorem \ref{t:f.g. for delta_g minimizer}, we know that the minimizer of $\tbeta_{X,\Delta}$ is an lc place of a $\mathbb{Q}$-complement. 
Then as showed in \cite{BLX19}*{Theorem 3.5},
$$h(X_t,\Delta_t)=\min_v\{\tbeta_{X_t,\Delta_t} \ | \ \mbox{$v$ is an $N$-complement}\},$$
for some constant $N$ which only depends on $\dim X$ and coefficients of $\Delta$. Then the rest of the proof is similar to the one in \cites{BL18, BLX19}. For the sake of completeness, we give a sketch here.

We know that there is a finite type variety $\phi\colon S\to B$ with a relative Cartier divisor $D\subset X\times_BS$ over $S$, such that 
\begin{enumerate}
\item for any $s\in S$, the fiber $D_s$ is an $N$-complement of $(X_t, \Delta_t)$ where $t=\phi(s)$, and $(X_t, D_s+\Delta_t)$ is log canonical but not klt; and
\item for any $N$-complement $\Gamma_t$ of $(X_t,\Delta_t)$, there is a point $s\in S$, such that $D_s\cong \Gamma_t$. 
\end{enumerate}
After resolving and stratifying $S$, as well as passing to a finite base change, we can assume $S$ is a  union of its smooth connected component $S_i$,
such that $(X\times_BS_i, \Delta\times_BS_i+D\times_SS_i)$ admits a fiberwise log resolution. 

For a fixed $i$, we can identify the dual complex $\mathcal{CW}_i:=\DMR(X_t, \Delta_t+D_t)$ for any $t\in S_i$. We claim $\tbeta_{X_t,\Delta_t}(v_t)$ does not depend on $t$, for different valuations $v_t$ correspond to the same point of $\mathcal{CW}_i$. This is obvious for $A_{X_t,\Delta_t}(v_t)$. It also proved in \cite{BLX19}, using the invariance of plurigenera (\cite{HMX13}),  for  $v_t$ corresponding to the same point  of $\mathcal{CW}$ over any $t\in S_i$, the induced DH-measure $\Dh_{\cF_{v_t}}$ on $\mathbb R$ is the same. Therefore, 
$$\tS(\cF_{v_t}) =  - \log \int_{\R}  e^{- \la} \, \Dh_{\cF_{v_t}}(d \la)$$ does not depend on $t$.

Hence, for each $i$, we can define $a_i=\min \{ \tbeta(v_t)\ |\  v_t\in \mathcal{CW}_i\}$, and we know
that $h(X_t,\Delta_t)=\min \{a_i | \ t\in \phi(S_i)\},$ which implies that $h(X_t,\Delta_t)$ is constructible. 

\bigskip

In light of the above constructibility result, to prove the lower semicontinuity of $h$ it suffices to consider the case when   $B$ is the spectrum of a DVR $R$ essentially of finite type over $\bk$. 
Let $K$ denote the fraction field of $R$ and  $\kappa$ the residue field.
By the properness of the flag variety, we know that any filtration $\cF_K$ on $\bigoplus_mH^0(-mr(K_{X_K}+\Delta_{X_K}))$ extends to a filtration $\cF_\kappa $ on $\bigoplus_mH^0(-mr(K_{X_\kappa}+\Delta_{X_\kappa}))$ (see \cite{BL18}). By Lemma \ref{l:L^NA=lc slope} and the lower semicontinuity of the log canonical threshold,
$$\LNA(\cF_K)= \mu(\cF_K) \ge \mu(\cF_\kappa) =  \LNA(\cF_\kappa).$$
Since $\Dh_{\cF_K}(d \la)= \Dh_{\cF_\kappa }(d \la)$, we also have
$\widetilde{S}(\cF_K)= \widetilde{S}(\cF_\kappa)$.
Therefore, $h$ is lower semicontinuous.
\end{proof}

\end{document}